\documentclass[reqno,11pt]{amsart}
\usepackage[T1]{fontenc}

\usepackage[utf8]{inputenc}
\usepackage[english]{babel}
\usepackage{url}
\usepackage{comment}
\usepackage{xcolor,color}
\usepackage{geometry}
\usepackage[all]{xy}
\usepackage{booktabs}
\usepackage{slashed}
\usepackage{bbm}
\usepackage{cancel}
\usepackage{adjustbox}
\usepackage[inline]{enumitem}
\usepackage{xparse}

\usepackage{tikz}
\usepackage{tikz-cd}

\usepackage{soul}

\usepackage{amsmath}
\usepackage{amssymb,graphicx}
\usepackage{amsthm}
\usepackage{latexsym}
\usepackage{amsfonts}
\usepackage{mathrsfs}

\usepackage{etex}
\usepackage[colorlinks = true, linkcolor = blue, urlcolor = black, citecolor = blue, anchorcolor = blue]{hyperref}
\usepackage{cleveref}

\numberwithin{equation}{section}

\theoremstyle{plain}
\newtheorem{lemma}{Lemma}[section]

\newtheorem{proposition/definition}[lemma]{Proposition/Definition}
\newtheorem{theorem}[lemma]{Theorem}
\newtheorem{corollary}[lemma]{Corollary}

\theoremstyle{definition}
\newtheorem{definition}[lemma]{Definition}
\newtheorem{example}[lemma]{Example}

\theoremstyle{remark}
\newtheorem{remarkInner}[lemma]{Remark}

\NewDocumentEnvironment{remark}{o}{%
  \pushQED{\qed}
  \IfValueTF{#1}{\remarkInner[#1]}{\remarkInner}%
}{%
  \popQED\endremarkInner
}

\newcommand{\calD}{\mathcal{D}}
\newcommand{\calE}{\mathcal{E}}

\newcommand{\calL}{\mathcal{L}}

\newcommand{\calO}{\mathcal{O}}

\newcommand{\calT}{\mathcal{T}}

\newcommand{\bbI}{\mathbb{I}}

\newcommand{\bbR}{\mathbb{R}}

\allowdisplaybreaks
\usepackage{cancel}

\title[Homogeneous K\"ahler Manifolds]{On Homogeneous K\"ahler Manifolds}

\author{Antonio De Nicola}
\address{Dipartimento di Matematica, Universit\`a degli Studi di Salerno, via Giovanni Paolo II n${}^\circ$ 132, 84084 Fisciano (SA), Italy}
\email{\href{mailto:antondenicola@gmail.com}{\underline{\smash{antondenicola@gmail.com}}}}

\author{Fabrizio Pugliese}
\address{Dipartimento di Matematica, Universit\`a degli Studi di Salerno, via Giovanni Paolo II n${}^\circ$ 132, 84084 Fisciano (SA), Italy}
\email{\href{mailto:fpugliese@unisa.it}{\underline{\smash{fpugliese@unisa.it}}}}

\author{Luca Vitagliano}
\address{Dipartimento di Matematica, Universit\`a degli Studi di Salerno, via Giovanni Paolo II n${}^\circ$ 132, 84084 Fisciano (SA), Italy}
\email{\href{mailto:lvitagliano@unisa.it}{\underline{\smash{lvitagliano@unisa.it}}}}

\keywords{}

\makeatletter
\def\l@subsection{\@tocline{2}{0pt}{2.5pc}{3.5pc}{}}
\makeatother

\begin{document}

\begin{abstract}
The cone $M \times \bbR_+$ over a Sasakian manifold $M$ is equipped with a canonical K\"ahler structure with specific homogeneity properties with respect to the $\bbR_+$ coordinate. This K\"ahler structure completely encodes the underlying Sasakian structure. Recently, Grabowski, Grabowska and Mohseni provided a broader conceptual framework for this phenomenon via \emph{homogeneous K\"ahler structures}, i.e.~K\"ahler structures on a principal $\bbR^\times$-bundle $P$ satisfying similar homogeneity properties. This approach successfully extends Sasakian geometry from cooriented contact manifolds (where $P$ is a trivial principal bundle) to non-necessarily coorientable contact structures (where $P$ is non-necessarily trivial). Homogeneous K\"ahler structures are genuinely more general than Sasakian structures and this note precisely characterizes the extent of this generalization. This is achieved through a detailed analysis of all the involved compatibilities in terms of the line bundle tautologically associated to $P$. We also show that modifying the homogeneity condition on the K\"ahler structure allows this framework to encompass co-K\"ahler structures and a natural generalization of those as well. 
\end{abstract}

\maketitle

\tableofcontents


\section*{Introduction}

There exists a dictionary that allows one to translate from Symplectic-related Geometries to Contact-related Geometries (see, e.g., \cite{PSV26} and references therein). Remarkably, this dictionary can also be applied to Complex Geometry providing, as an outcome, a version of the odd dimensional counterpart of Complex Geometry: the so called Almost Contact Geometry \cite{B02}. Even more, this dictionary can be applied to any $G$-structure \cite{TVY20}. The Symplectic-to-Contact Dictionary comes in two versions that we now briefly explain. Let $M$ be a smooth manifold and let $H \subseteq TM$ be a contact distribution on $M$ (i.e.~a maximally non-integrable hyperplane distribution). The quotient bundle $L := TM/H$ is a line bundle, and the canonical projection $\beta \colon TM \to L$ can be seen as an $L$-valued $1$-form on $M$: a generalized version of a contact form. Consider the \emph{Atiyah algebroid} $\mathsf D L \to M$ of $L$, whose sections are derivations of $L$, i.e.~$\bbR$-linear maps $D \colon \Gamma (L) \to \Gamma (L)$ satisfying the following Leibniz rule: $D (f\lambda) = \sigma (D) (f) \lambda + f D \lambda$, for all $f \in C^\infty (M)$ and all $\lambda \in \Gamma (L)$. Here $\sigma (D)$ is a vector field on $M$ uniquely determined by $D$ and sometimes called the \emph{symbol} of $D$. The Atiyah algebroid is a Lie algebroid with bracket given by the commutator of derivations and anchor given by the symbol map $\sigma \colon \mathsf DL \to TM$. This Lie algebroid is equipped with a tautological representation on $L$: the action of a derivation on a section of $L$. Composing $\beta$ with the symbol, we get a $1$-cochain on $\mathsf DL$ with coefficients in its tautological representation $L$: $B = \beta \circ \sigma \colon \mathsf D L \to L$. The de Rham differential of $B$ is a \emph{symplectic Atiyah form}, i.e.~a $2$-cocycle $\Omega$ on $\mathsf D L$ with coefficients in $L$ satisfying an appropriate non-degeneracy property. It can be proved that the assignment $(M, H) \mapsto (L, \Omega)$ establishes an equivalence between the category of contact manifolds and contactomorphisms and the category of line bundles $L$ equipped with a symplectic Atiyah form and line bundle isomorphisms preserving $\Omega$ (see \cite{PSV26} for details). This suggests that, to translate constructions and results from Symplectic to Contact Geometry, we can perform the following substitutions:
\[
\begin{aligned}
C^\infty (M) & \leadsto \Gamma (L) \\
TM & \leadsto \mathsf DL \\
\text{differential forms on $M$} & \leadsto \text{$L$-valued forms on the Atiyah algebroid}
\end{aligned}
\] 
This dictionary has been successfully applied in many different situations (see, e.g., \cite{V15, V18, VW16, VW20, VW20bis, SV20, BTV20, S23, ST23, MTV24}).

There is another version of the dictionary. The line bundle $L$ defines a principal $\bbR^\times$-bundle $\widetilde L$, where $\bbR^\times$ is the multiplicative Lie group of non-zero reals, $\widetilde L = L^\ast \smallsetminus 0$ and the principal action $h \colon \bbR^\times \times \widetilde L \to \widetilde L$ is given by fiberwise scalar multiplication. Denote by $\pi \colon \widetilde L \to M$ the bundle projection. The ``contact form'' $\beta \colon TM \to L$ can be seen as a semibasic $1$-form $\widetilde B$ on $\widetilde L$ by putting $\widetilde B (v) = \langle \varepsilon, \beta (\pi_\ast v) \rangle$ for all $v \in T_\varepsilon \widetilde L$, and $\varepsilon \in \widetilde L$, where $\langle -, -\rangle \colon L^\ast \times_M L \to \bbR$ denotes the duality pairing.

The de Rham differential of $\widetilde B$ is a closed $2$-form on $\widetilde L$. The assignment $(M, H) \mapsto (\widetilde L, \widetilde \Omega)$ establishes an equivalence between the category of contact manifolds and contactomorphisms and the category of homogeneous symplectic manifolds and homogeneous symplectomorphisms. Here, by a \emph{homogeneous symplectic manifold} we mean a principal $\bbR^\times$-bundle $P$, with principal action $h$, and a symplectic form $\widetilde \Omega \in \Omega^2 (P)$ satisfying the following \emph{homogeneity condition}: 
\begin{equation}\label{eq:hs}
h_r^\ast \widetilde \Omega = r\widetilde \Omega
\end{equation}
 for all $r \in \bbR^\times$. Accordingly, a \emph{homogeneous symplectomorphism} is an $\bbR^\times$-equivariant symplectomorphism of homogeneous symplectic manifolds. This suggests that, to translate from Symplectic to Contact Geometry it is enough to pass from manifolds to principal $\bbR^\times$-bundles, aka \emph{homogeneous manifolds}, and implement the appropriate homogeneity condition. This version of the dictionary has been successfully applied to different situations mainly by J.~Grabowski and collaborators (see \cite{G13,GG22,GG23,GGM25} and references therein). Notice that, for structures different from symplectic forms, the appropriate homogeneity condition might look very different from \eqref{eq:hs}. 
 
 The two versions of the Symplectic-to-Contact dictionary discussed so far are actually related. Namely, the assignment $L \mapsto \widetilde L$ establishes an equivalence between the category of line bundles (with fiber-wise invertible line bundle maps) and the category of homogeneous manifolds (with $\bbR^\times$-equivariant maps). Accordingly, natural constructions with line bundles correspond to natural constructions with homogeneous manifolds. For instance, there is a \emph{homogenization isomorphism} $\lambda \mapsto \widetilde \lambda$ mapping bijectively section of $L$ to degree $1$ homogeneous functions on $\widetilde L$, i.e.~functions $f \in C^\infty (\widetilde L)$ such that $h_r^\ast f = r f$ for all $r \in \bbR^\times$. Similarly, there are homogenization isomorphisms
 \[
 \begin{aligned}
 D & \mapsto \widetilde D \\
 \Omega & \mapsto \widetilde \Omega
 \end{aligned}
 \]
 mapping bijectively derivations $D$ of $L$ to degree $0$ homogeneous vector fields $X$ on $\widetilde L$ (i.e., $h^\ast_r X = X$) and $L$-valued cochains $\Omega \colon \wedge^\bullet \mathsf D L \to L$ on the Atiyah algebroid to degree $1$ homogeneous differential forms $\omega$ on $\widetilde L$ (i.e., $h^\ast_r \omega = r \omega$). For more details, see \cite{VW20bis, PSV26}. The symplectic Atiyah form $\Omega$ and the homogeneous symplectic form $\widetilde \Omega$ determined by a contact structure as explained above are actually related by the latter homogenization (and this explains the notation that we have adopted).

Now, let $P \to M$ be a principal $\bbR^\times$-bundle with principal action $h$. In \cite{GGM25} the authors define a \emph{homogeneous K\"ahler structure} on $P$ as a K\"ahler structure $(\widetilde K, \widetilde G, \widetilde \Omega)$ on $P$, with $\widetilde K$ being the complex structure, $\widetilde G$ being the Riemannian metric, and $\widetilde \Omega$ being the K\"ahler form, satisfying the following homogeneity conditions:
\begin{equation}\label{eq:hK0}
\begin{aligned}
h_r^\ast \widetilde K &= \operatorname{sign}(r) \widetilde K\\
h_r^\ast \widetilde G & = |r| \widetilde G \\
h_r^\ast \widetilde \Omega & = r \widetilde \Omega,
\end{aligned}
\end{equation}
for all $r \in \bbR^\times$. The third one of \eqref{eq:hK0} does actually follow from the previous two. These apparently weird equations are actually partly forced by $\widetilde K$ satisfying $\widetilde K{}^2 = -1$ and by $\widetilde G$ being positive definite. Under certain conditions a homogeneous K\"ahler structure on $P$ determines a Sasakian structure on $M$, but the former are genuinely more general than the latter in two directions: 1) The contact structure on $M$ determined by the homogeneous symplectic part $\widetilde \Omega$ of the homogeneous K\"ahler structure need not be co-oriented, not even co-orientable, and 2) $(\widetilde K, \widetilde G, \widetilde \Omega)$ determine an honest Sasakian structure on $M$ only when the orthogonal distribution to the fibers of $P \to M$ is integrable. So the Symplectic-to-Contact dictionary does not strictly apply in Sasakian geometry, in that Sasakian manifolds are \emph{not} equivalent to homogeneous K\"ahler manifolds. The emphasis in \cite{GGM25} is on aspect 1), while the emphasis in this note is on aspect 2). More precisely, in the present note
\begin{enumerate}
\item we characterize the exact extent to which homogeneous K\"ahler structures are more general than Sasakian structures,
\item using the equivalence between line bundles and homogeneous manifolds, we re-interpret the results in \cite{GGM25} in terms of line bundles.
\end{enumerate}
We do this studying systematically all the compatibilities between the K\"ahler tensors $\widetilde K, \widetilde G, \widetilde \Omega$ in terms of line bundles. In this way, 
\begin{enumerate}
\item[(3)] we also apply the Symplectic-to-Contact dictionary to (partially) non-integrable versions of K\"ahler Geometry, namely Almost Hermitian, Hermitian, and Almost K\"ahler Geometries.
\end{enumerate}
Finally, Equations \eqref{eq:hK0} are not the only possible homogeneity conditions that one can impose on a K\"ahler structure (more generally an almost Hermitian structure) on a homogeneous manifold. So
\begin{enumerate}
\item[(4)] we study all possible homogeneity conditions on an almost Hermitian structure on a homogeneous manifold and show that there are essentially only three inequivalent ones: the first one is related to Sasakian Geometry as in \cite{GGM25}, the second one is related to co-K\"ahler geometry in a similar way, while the third one is just a twisted version of the second (see the appendix).
\end{enumerate}

The paper is divided in three sections and one appendix.

In Section \ref{sec:1} we review the equivalence between line bundles and homogeneous manifolds with an emphasis on calculus on a homogeneous manifold and all its associated line bundles. We conclude this section showing that there are essentially only three meaningful homogeneity conditions that one can impose on an almost Hermitian structure on a homogeneous manifold. We call the first two the \emph{homogeneous} and \emph{invariant} cases.

In Section \ref{sec:2} we study the homogeneous case in details. There is an overlap with \cite{GGM25} here, but our emphasis is on finding a line bundle description of homogeneous almost Hermitian structure by studying separately 1) all the involved algebraic conditions on one side, and 2) all the integrability conditions defining an Hermitian, an almost K\"ahler and a K\"ahler structure respectively. It turns out that homogeneous K\"ahler manifolds are more general than Sasakian manifolds, even locally. We also discuss which precise homogeneous K\"ahler structures are equivalent to a Sasakian structure (on the base). We conclude Section \ref{sec:2} with an example of a homogeneous K\"ahler manifold with trivial associated line bundle not coming from a Sasakian structure on the base.

In Section \ref{sec:3} we study the invariant case. This case is not covered by \cite{GGM25}. Our analysis here closely parallels that in Section \ref{sec:2}. Invariant K\"ahler structures are related to co-K\"ahler structures but are more general than the latter, even locally. We discuss which precise invariant K\"ahler structures are equivalent to a co-K\"ahler structure and provide two examples of an invariant K\"ahler manifold with trivial line bundle not coming from a co-K\"ahler structure in two different ways.

In the appendix we shortly discuss the remaining case.

We assume the reader is familiar with the fundamentals of Lie algebroids. We also assume familiarity with the basic definitions in \emph{Almost Contact Metric Geometry}, including the notions of \emph{normal almost contact metric manifold}, and \emph{contact metric / Sasakian manifold} (resp.~\emph{almost co-K\"ahler / co-K\"ahler manifold}). We only recall here that the latter is an odd dimensional analogue of an Hermitian manifold, and an almost K\"ahler / K\"ahler manifold, respectively (for more details, see \cite{B02}).

All base manifolds $M$ of a principal $\bbR^\times$-bundle $P \to M$ will be assumed to be connected.

\subsection*{Acknowledgments} The authors are members of the GNSAGA of INdAM.

\section{Homogeneous Manifolds and Line Bundles}\label{sec:1}

\subsection{Homogenization}

Let $\bbR^\times$ be the multiplicative group of non-zero reals. We will often interpret $\mathbb R^\times$ as the order $1$ general linear group $\operatorname{GL}(1,\mathbb R)$. As such, it comes with a tautological $1$-dimensional representation on $\mathbb R$. Let $P\rightarrow M$ be a principal $\mathbb R^\times$-bundle. Denote by $\pi\colon P\rightarrow M$ the bundle projection, and by $h\colon\bbR^\times\times P\rightarrow P,\ (r,p)\mapsto h_r(p)$, the principal $\bbR^\times$-action. In this paper, we will refer to the data $(P \to M, h)$ as a \textbf{homogeneous manifold}. The reason behind this terminology is that we use $h$ to detect \emph{homogeneity properties} of a tensor on $P$ with respect to the $\pi$-fiber coordinate.

The base manifold $M = P/\bbR^\times$ will be sometimes denoted $P/h$ to emphasize the action that we are quotienting out. A \textbf{homogeneous map} is an $\mathbb R^\times$-equivariant map of homogeneous manifolds. Homogeneous manifolds and homogeneous maps form a category. There are obvious mutually quasi-inverse category equivalences between homogeneous manifolds with homogeneous maps and line bundles with fiber-wise invertible vector bundle maps. Under this equivalences a homogeneous manifold $P$ corresponds to the line bundle $(P \times \mathbb R) / \mathbb R^\times \to M$ associated to the tautological representation of $\mathbb R^\times$. Conversely, a line bundle $L \to M$ corresponds to the homogeneous manifold $\widetilde L = L^\ast \smallsetminus 0$, where the principal $\mathbb R^\times$-action is given by fiber-wise scalar multiplication.

We now recall from \cite{VW20, TVY20, PSV26} some implications of the equivalence between line bundles and homogeneous manifolds. We also make a number of easy remarks on line bundles associated to a homogeneous manifold which will be useful in the rest of the paper.

We begin discussing derivations of a (generically higher rank) vector bundle $E \to M$. A \emph{derivation} of $E$ is an $\mathbb R$-linear map $D\colon \Gamma (E) \to \Gamma (E)$ satisfying the following Leibniz rule: for all $e \in \Gamma (E)$ and all $f \in C^\infty (M)$,
\[
D(fe) = f D(e) + \sigma_D (f) e,
\]
where $\sigma_D \in \mathfrak X (M)$ is some (necessarily unique) vector field called the \emph{symbol} of $D$. Derivations of $E \to M$ can be seen as sections of the \textbf{gauge algebroid} (or \textbf{Atiyah algebroid}) $\mathsf{D}E \Rightarrow M$ of $E$, whose Lie bracket is the commutator of derivations and whose anchor is the symbol map $\sigma \colon \mathsf{D}E \to TM$, $D \mapsto \sigma_D$. We denote $\calD (E) = \Gamma (\mathsf DE)$.



We now come to line bundles. So let $L \to M$ be a line bundle. There are at least two (actually several) de Rham-like complexes associated to $L$. Namely, the gauge algebroid $\mathsf D L$ acts tautologically on $L$ and trivially on $\mathbb R_M$. These two representations are encoded by complexes $(\Omega^\bullet_{\mathsf D L} (L), d_{\mathsf DL})$ and $(\Omega^\bullet_{\mathsf{D}L}, d_{\mathsf DL})$ consisting of alternating $C^\infty (M)$-multilinear maps $\calD (L) \times \cdots \times \calD (L) \to \Gamma (L)$ and $\calD (L)\times \cdots \times \calD (L) \to C^\infty (M)$ respectively. Cochains in $\Omega^\bullet_{\mathsf D L} (L)$ will be referred to as \textbf{Atiyah forms} on $L$. 


%

Now, let $(P \to M, h)$ be a homogeneous manifold and let $L \to M$ be the associated line bundle. We denote by $\pi$ both the projections $P \to M$ and $L \to M$, and identify $P$ with $\widetilde L = L^\ast \smallsetminus 0$. The vector bundle $T^{p,q}P \to P$ of $(p,q)$-tensors on $P$ ($p$-times covariant, $q$-times contravariant) is a \emph{homogeneous vector bundle} (i.e.~an equivariant vector bundle over a homogeneous manifold) in several different ways. Namely, for every Lie group homomorphism $f \colon \bbR^\times \to \bbR^\times$, there is a principal action by vector bundle automorphisms covering $h$:
\[
h^f \colon \mathbb R^\times \times T^{p,q}P \to T^{p,q}P, \quad (r, \calT) \mapsto f(r) h_{r\ast} (\calT).
\]

\begin{remark}
There are exactly two families of Lie group homomorphisms $f \colon \bbR^\times \to \bbR^\times$ both parameterized by a real number $\alpha$, namely
\[
f(r) = |r|^\alpha \quad \text{and} \quad f(r) = \operatorname{sign}(r) |r|^\alpha.
\]
\end{remark}

Homogeneous (i.e.~$\mathbb R^\times$ equivariant) sections of $T^{p,q}P \to P$ with respect to $h^f$ will be also called \emph{$f$-homogeneous}. When $f(r) = r^k$ for some integer $k$, $f$-homogeneous tensors are also called \emph{degree $k$ homogeneous}, or simply \emph{$k$-homogeneous}. In other words a $(p,q)$-tensor field $\mathcal T \in \Gamma (T^{p,q} P)$ is $k$-homogeneous if $h_r^\ast (\mathcal T) = r^k \mathcal T$ for all $r \in \mathbb R^\times$.

Denote by $\mathsf D^\ast L$ the dual vector bundle of $\mathsf DL$. Then, the quotient vector bundle $T^{p,q} P/ h^f$ is $f(L) \otimes T^{p,q}_{\mathsf DL}$ where $T^{p,q}_{\mathsf DL} := (\mathsf DL)^{\otimes q} \otimes (\mathsf D^\ast L)^{\otimes p}$. Here $f(L)$ is the line bundle associated to the representation of $\mathbb R^\times$ on $\bbR$ encoded by $f$ (equivalently, it is the line bundle whose transition maps are obtained from those of $L$ by postcomposing with $f$). Accordingly, we have a pull-back diagram 
\begin{equation*}
	\begin{tikzcd}
		T^{p,q}P\arrow[rr]\arrow[d]&&f(L) \otimes T^{p,q}_{\mathsf DL} \arrow[d]\\
		P\arrow[rr, "\pi"]&&M
	\end{tikzcd},
\end{equation*}
so that pull-back sections are precisely $f$-homogeneous tensor fields. For every section $T$ of  $f(L) \otimes T^{p,q}_{\mathsf DL} \to M$, we also denote by $\widetilde T = \pi^\ast T$ the corresponding $f$-homogeneous tensor field and call it the \textbf{homogenization} of $T$. Summarizing, the assignment $T \mapsto \widetilde T$ establishes a bijection between sections of $f(L) \otimes T^{p,q}_{\mathsf DL} \to M$ and $f$-homogeneous tensor fields on $P$, which we will use throughout the paper.

\begin{remark}
In the following list we specialize the latter remarks to some situations of interest for the purposes of this paper:

\begin{itemize}
\item[\checkmark] When $p,q, k = 0$, $T = f$ is a function on $M$, $\widetilde f = \pi^\ast f$ is a basic function on $P$ and the homogenization $f \mapsto \widetilde f$ identifies smooth functions on $M$ and $0$-homogeneous functions on $P$.

\item[\checkmark] When $p,q = 0$ and $k = 1$, $T = \lambda$ is a section of $L$, $\widetilde \lambda$ is given by $\widetilde \lambda (p) = \langle p, \lambda_x \rangle$, for all $p \in P_x = L^\ast_x \smallsetminus 0$, $x \in M$, and the homogenization $\lambda \mapsto \widetilde \lambda$ identifies sections of $L$ and $1$-homogeneous functions on $P$.


\item[\checkmark] When $p = 0, q = 1$ and $k = 0$, $T = D$ is a derivation of $L$, and the homogenization $D \mapsto \widetilde D$ identifies derivations of $L$ and $0$-homogeneous vector fields on $P$.

\item[\checkmark] When $q = 0, k = 0$ and $T = \Omega$ is skewsymmetric, then $\Omega \in \Omega^p_{\mathsf DL}$, $\widetilde \Omega \in \Omega^p (P)$, and the homogenization $\Omega \mapsto \widetilde \Omega$ identifies alternating $p$-forms on $\mathsf D L$ and $0$-homogeneous differential $p$-forms on $P$.

\item[\checkmark] When $q = 0, k = 1$ and $T = \Omega$ is skewsymmetric, then $\Omega \in \Omega^p_{\mathsf DL} (L)$, $\widetilde \Omega \in \Omega^p (P)$, and the homogenization $\Omega \mapsto \widetilde \Omega$ identifies Atiyah $p$-forms on $L$ and $1$-homogeneous differential $p$-forms on $P$. 
\end{itemize}
\end{remark}

\begin{example}
The homogenization of the identity derivation $\bbI \in \calD (L)$ is the (restriction to $\widetilde L = L^\ast \smallsetminus 0$ of) the Euler vector field $\calE$ (on $L^\ast$): $\widetilde \bbI = \calE$. In what follows we will freely use this simple remark.
\end{example}

More cases will appear in the sequel.
\subsection{Associated Line Bundles}\label{subsec:AssLB}

Homogenization is compatible with all natural pointwise operations on tensors. For instance, $\widetilde{\Omega(D_1, \ldots, D_p)} = \widetilde \Omega (\widetilde D_1, \ldots, \widetilde D_p)$ for all $\Omega \in \Omega^p_{\mathsf{D} L}$ (resp.~$\Omega \in \Omega^p_{\mathsf{D} L} (L)$), and $D_1, \ldots, D_p \in \mathcal D (L)$. Homogenization is also compatible with all natural differential operators on tensors. For instance, $\widetilde{[D_1, D_2]} = [\widetilde D_1, \widetilde D_2]$ for all $D_1,D_2 \in \calD (L)$. Similarly, $\widetilde{d_{\mathsf D L} \Omega} = d \widetilde \Omega$ for all $\Omega \in \Omega^p_{\mathsf{D} L}$ (resp.~$\Omega \in \Omega^p_{\mathsf{D} L} (L)$). In particular, for any two Lie group homomorphisms $f,f' \colon \bbR^\times \to \bbR^\times$:

\begin{itemize}
\item[\checkmark] there is a unique bilinear map (actually a vector bundle isomorphism)
\[
f(L) \otimes f'(L) \to f\!f' (L), \quad \lambda \otimes \lambda' \mapsto \lambda \lambda'
\]
such that $\widetilde{\lambda \lambda'} = \widetilde \lambda \, \widetilde \lambda'$;
\item[\checkmark] for all $p,q$, there is a unique $\bbR$-bilinear map (also called \textbf{Lie derivative})
\[
\Gamma \Big(f(L) \otimes \mathsf DL\Big) \times \Gamma \Big( f'(L) \otimes T^{p,q}_{\mathsf DL} \Big) \to \Gamma \Big( f\!f'(L) \otimes T^{p,q}_{\mathsf DL}\Big), \quad (D, \calT) \mapsto \calL_D \calT
\]
such that $\widetilde{\calL_D \calT} = \calL_{\widetilde D} \widetilde \calT$.
\end{itemize}
Here $f\!f'$ is the point-wise multiplication of $f, f'$. From the above, more natural operations arise, satisfying the obvious compatibilities with homogenization. For instance, given $K \in \Gamma \big(f(L) \otimes T^{1,1}_{\mathsf DL}\big)$, one can uniquely define the \textbf{Nijenhuis torsion} $N_K \in \Gamma \big(f^2(L) \otimes T^{2,1}_{\mathsf DL}\big)$ of $K$ so that $\widetilde{N_K}$ is exactly the Nijenhuis torsion of $\widetilde K$, etc. In what follows, we will freely use all these natural operations. If the reader feels uncomfortable with these constructions, they should just think of sections of $f(L) \otimes T^{p,q}_{\mathsf DL}$ as ordinary $(p,q)$ tensor fields on $P = \widetilde L$ satisfying the appropriate homogeneity condition.

\begin{remark}
In what follows we will always implicitly use the natural isomorphism $f(L) \otimes f'(L) \cong f\!f'(L)$ to identify the two line bundles. For instance, we will always identify $L^\ast \otimes L= L^{-1} \otimes L$ and the trivial line bundle $L^0 = \bbR_M$. Similarly, we will identify $\operatorname{sign}(L)^{\otimes 2}$ and $\bbR_M$.
\end{remark}

\begin{remark}\label{rem:|-|,sign}
The line bundles $|L|$ and $\operatorname{sign} (L)$ play a special role in this paper. For this reason, we recall here some of their properties that will be used later on. First of all $\operatorname{sign} (L) \otimes |L| = L$. Moreover $|L|$ is orientable (hence trivial) and canonically oriented (but not canonically trivial). A section $\lambda \in \Gamma (|L|)$ is positive if and only if its homogenization $\widetilde \lambda$ is a positive ($|-|$-homogeneous) function on $\widetilde L$. Finally, $\operatorname{sign} (L)$ possesses a canonical flat connection which we will denote $\nabla^{can}$. A section of $\lambda \in \Gamma (\operatorname{sign} (L))$ is $\nabla^{can}$-parallel if and only if its homogenization $\widetilde \lambda$ is a locally constant ($\operatorname{sign}$-homogeneous) function on $\widetilde L$. In what follows we will denote by $d^{can} : \Omega^\bullet (M, \operatorname{sign}(L)) \to \Omega^{\bullet +1} (M, \operatorname{sign}(L))$ the connection differential corresponding to $\nabla^{can}$. When $L = \bbR_M$ is a trivial line bundle, so is $\operatorname{sign}(L)$ and $d^{can}$ is just the ordinary de Rham differential. 
\end{remark}

\begin{remark}
Let $f \colon \bbR^\times \to \bbR^\times$ be a Lie group homomorphism. A derivation $D$ of $L$ (or, equivalently, a $0$-homogeneous vector field $\widetilde D$ on $\widetilde L$) determines a unique derivation $D^f$ of the associated line bundle $f(L)$ via
\[
\widetilde{D^f \lambda} = \widetilde D \widetilde \lambda, \quad \lambda \in \Gamma \big(f(L)\big).
\]
The assignment $\mathsf D L \to \mathsf D f(L)$ is a canonical flat $\mathsf D L$-connection in $f(L)$. When
\[
\dot f := \frac{df(r)}{dr}|_{r = 1} \neq 0,
\]
i.e. $f(r) = |r|^\alpha, \operatorname{sign}(r)|r|^\alpha$ with $\alpha \neq 0$, then the connection $\mathsf D L \to \mathsf D f(L)$ is actually a Lie algebroid isomorphism that we will often use to identify $\mathsf D L$ and $\mathsf D f(L)$ (together with their representations on associated bundles). In this case we simply write $D$ for $D^f$.

When $\dot f = 0$, i.e. $f(r) = 1, \operatorname{sign} (r)$, then the connection $\mathsf D L \to \mathsf D f(L)$ is not an isomorphism and it  actually factorizes as
\[
\mathsf D L \overset{\sigma}{\longrightarrow}  TM \overset{\nabla^{can}}{\longrightarrow} \mathsf D f(L)
\]
where $\nabla^{can} \colon TM \to \mathsf D f(L)$ is the canonical flat connection in either $\bbR_M$ or $\operatorname{sign} (L)$.
\end{remark}

\begin{remark}\label{rem:L-conn}
A linear connection $\nabla$ in $L$ or, equivalently, a principal connection in $\widetilde L$,  induces a linear connection $\nabla^f$ in every associated (line) bundle $f(L)$. The connection $\nabla^f$ is given by the composition
\[
TM \overset{\nabla}{\longrightarrow} \mathsf D L \longrightarrow \mathsf D f (L).
\]
When $\dot f \neq 0$, the curvatures of $\nabla, \nabla^f$, seen as closed $2$-forms on $M$, agree, and we simply write $\nabla$ for $\nabla^f$. When $\dot f = 0$, then $\nabla^f$ agrees with the canonical flat connection $\nabla^{can}$.
\end{remark}

\subsection{Almost Hermitian Structures on Homogeneous Manifolds}\label{subsec:AHSonHM}

We conclude this section discussing the possible homogeneity conditions for an almost Hermitian structure on a homogeneous manifold $P = \widetilde L$. So let $(\widetilde K, \widetilde G, \widetilde \Omega)$ be an \emph{almost Hermitian structure} on $\widetilde L$, i.e.~$\widetilde K$ is an almost complex structure, $\widetilde G$ is a Riemannian metric, and $\widetilde \Omega$ is a non-degenerate $2$-form on $\widetilde L$ satisfying the usual compatibility
\begin{equation}\label{eq:comp_0}
\widetilde \Omega = \widetilde G(-, \widetilde K -).
\end{equation}
First of all, from $\widetilde K{}^2 = -1$ it follows that $\widetilde K$ can only be $f$-homogeneous for $f(r) = \operatorname{sign}(r)$ or $f(r) = 1$. Similarly, from $\widetilde G$ being positive definite it follows that $\widetilde G$ can only be $f'$-homogeneous for $f'(r) = |r|^\alpha$, $\alpha \in \bbR$. Finally, if $\widetilde K$ is $f$-homogeneous, and $\widetilde G$ is $f'$-homogeneous, then, from \eqref{eq:comp_0}, $\widetilde \Omega$ is necessarily $f\!f'$-homogeneous. We now distinguish the following cases:

\begin{itemize}
\item[] \hspace{-30pt} \textbf{CASE $\alpha$-A:} $f(r) = \operatorname{sign}(r)$ and $f'(r) = |r|^\alpha$, with $\alpha \neq 0$. In this case $\widetilde K, \widetilde G, \widetilde \Omega$ are the homogenizations of
\begin{enumerate}
\item[\checkmark] a section of $\operatorname{sign}(L) \otimes T^{1,1}_{\mathsf D L}$, equivalently a vector bundle map 
\[
K \colon \mathsf DL \to \operatorname{sign}(L) \otimes \mathsf D L,
\]
\item[\checkmark] a section of $|L|^\alpha \otimes T^{2,0}_{\mathsf D L}$, equivalently a $2$-form 
\[
G \colon \mathsf DL \otimes \mathsf DL \to |L|^\alpha,
\]
\item[\checkmark] a section of $\operatorname{sign}(L) \otimes |L|^\alpha \otimes T^{2,0}_{\mathsf D L}$, equivalently a $2$-form 
\[
\Omega \colon \mathsf DL \otimes \mathsf D L \to \operatorname{sign}(L) \otimes |L|^\alpha,
\]
\end{enumerate}
satisfying appropriate additional conditions (coming from $(\widetilde K, \widetilde G, \widetilde \Omega)$ being an almost Hermitian structure). If we put $L' = \operatorname{sign}(L) \otimes |L|^\alpha$, we have
\begin{itemize}
\item[\checkmark] $\mathsf D L' = \mathsf D L$,
\item[\checkmark] $\operatorname{sign} (L') = \operatorname{sign} (L)$,
\item[\checkmark]	$|L'| = |L|^\alpha$,
\end{itemize}
whence we can reinterpret $(K, G, \Omega)$ as vector bundle maps
\[
\begin{aligned}
K' & \colon \mathsf DL' \to \operatorname{sign}(L') \otimes \mathsf D L', \\
G' & \colon \mathsf DL' \otimes \mathsf DL' \to |L'|, \\
\Omega' & \colon \mathsf \mathsf DL' \otimes \mathsf DL' \to L'.
\end{aligned}
\]
It is easy to see that, more precisely, the homogenization $(\widetilde K{}', \widetilde G{}', \widetilde \Omega{}')$ of $(K', G', \Omega')$ is actually an almost Hermitian structure on $\widetilde L{}'$. This shows that CASEs $\alpha$-A are all basically equivalent to CASE $1$-A. In CASE $1$-A, we call $(\widetilde K, \widetilde G, \widetilde \Omega)$ a \textbf{homogeneous almost Hermitian structure}.

\bigskip

\item[] \hspace{-30pt} \textbf{CASE $\alpha$-B:} $f(r) = 1$ and $f'(r) = |r|^\alpha$, with $\alpha \neq 0$. In this case $\widetilde K, \widetilde G, \widetilde \Omega$ are the homogenizations of
\begin{enumerate}
\item[\checkmark] a section of $T^{1,1}_{\mathsf D L}$, equivalently a vector bundle map 
\[
K \colon \mathsf DL \to \mathsf D L,
\]
\item[\checkmark] a section of $|L|^\alpha \otimes T^{2,0}_{\mathsf D L}$, equivalently a $2$-form 
\[
G \colon \mathsf DL \otimes \mathsf DL \to |L|^\alpha,
\]
\item[\checkmark] a section of $|L|^\alpha \otimes T^{2,0}_{\mathsf D L}$, equivalently a $2$-form 
\[
\Omega \colon \mathsf DL \otimes \mathsf D L \to |L|^\alpha,
\]
\end{enumerate}
satisfying appropriate additional conditions. If we put $L' = |L|^\alpha$, we have
\begin{itemize}
\item[\checkmark] $\mathsf D L' = \mathsf D L$,
\item[\checkmark] $\operatorname{sign} (L') = \bbR_M$,
\item[\checkmark]	$|L'| = |L|^\alpha$,
\end{itemize}
whence we can reinterpret $(K, G, \Omega)$ as vector bundle maps
\[
\begin{aligned}
K' & \colon \mathsf DL' \to \operatorname{sign}(L') \otimes \mathsf D L', \\
G' & \colon \mathsf DL' \otimes \mathsf DL' \to |L'|, \\
\Omega' & \colon \mathsf DL' \otimes \mathsf DL' \to L',
\end{aligned}
\]
again. The homogenization $(\widetilde K{}', \widetilde G{}', \widetilde \Omega{}')$ of $(K', G', \Omega')$ is an almost Hermitian structure on $\widetilde L{}'$. This shows that CASEs $\alpha$-B all basically fall into CASEs $\alpha$-A, hence into the \emph{homogeneous case}: CASE $1$-A. Notice however that, in CASE $\alpha$-B, $L'$ is an orientable (hence trivializable) line bundle.

\bigskip

\item[] \hspace{-30pt} \textbf{CASE C:} $f(r) = \operatorname{sign}(r)$ and $f'(r) = 1$. In this case $\widetilde K, \widetilde G, \widetilde \Omega$ are the homogenizations of
\begin{enumerate}
\item a section of $\operatorname{sign}(L) \otimes T^{1,1}_{\mathsf D L}$, equivalently a vector bundle map 
\[
K \colon \mathsf DL \to \operatorname{sign}(L) \otimes \mathsf D L,
\]
\item a section of $T^{2,0}_{\mathsf D L}$, equivalently a $2$-form 
\[
G \colon \mathsf DL \otimes \mathsf DL \to \bbR_M,
\]
\item a section of $\operatorname{sign}(L) \otimes T^{2,0}_{\mathsf D L}$, equivalently a $2$-form 
\[
\Omega \colon \mathsf DL \otimes \mathsf D L \to \operatorname{sign}(L),
\]
\end{enumerate}
satisfying appropriate additional conditions. This case is not very different from the \emph{invariant case}, i.e.~next CASE D, and will be discussed in the appendix.

\bigskip

\item[] \hspace{-30pt}  \textbf{CASE D:} $f(r) = f'(r) = 1$. In this case $\widetilde K, \widetilde G, \widetilde \Omega$ are the homogenizations of
\begin{enumerate}
\item a section of $T^{1,1}_{\mathsf D L}$, equivalently a vector bundle map 
\[
K \colon \mathsf DL \to \mathsf D L,
\]
\item a section of $T^{2,0}_{\mathsf D L}$, equivalently a $2$-form 
\[
G \colon \mathsf DL \otimes \mathsf DL \to \bbR_M,
\]
\item a section of $T^{2,0}_{\mathsf D L}$, equivalently a $2$-form 
\[
\Omega \colon \mathsf DL \otimes \mathsf D L \to \bbR_M,
\]
\end{enumerate}
satisfying appropriate additional conditions. In CASE D, we call $(\widetilde K, \widetilde G, \widetilde \Omega)$ an \textbf{invariant almost Hermitian structure}.
\end{itemize}

The homogeneous and invariant cases are discussed in details in the next two sections.

\section{Homogeneous K\"ahler Manifolds}\label{sec:2}

\subsection{Homogeneous Almost Hermitian Structures}\label{subsec:2A}
Let $L \to M$ be a line bundle and let $(\widetilde L \to M, h)$ be the associated homogeneous manifold. A \textbf{homogeneous almost Hermitian structure} on $\widetilde L$ is an almost Hermitian structure $(\widetilde K, \widetilde G, \widetilde \Omega)$ on $\widetilde L$ satisfying the following homogeneity conditions (CASE $1$-A in Subsection \ref{subsec:AHSonHM}):
\begin{enumerate}
\item $h_r^\ast \widetilde K = \operatorname{sign}(r) \widetilde K$,
\item $h_r^\ast \widetilde G = |r| \widetilde G$, and
\item $h_r^\ast \widetilde \Omega = r \widetilde \Omega$,
\end{enumerate}
for all $r \in \bbR^\times$. It easily follows from the latter homogeneity conditions that $\widetilde K, \widetilde G, \widetilde \Omega$ are the homogenizations of
\begin{enumerate}
\item a vector bundle map $K \colon \mathsf DL \to \operatorname{sign}(L) \otimes \mathsf D L$,
\item a non degenerate, symmetric $2$-form $G \colon S^2\mathsf DL \to |L|$,
\item a non-degenerate, alternating $2$-form $\Omega \colon \wedge^2 \mathsf DL \to L$,
\end{enumerate}
satisfying
\begin{itemize}
\item[\checkmark] $K^2 = -1$, where $K^2$ is the composition
\[
\mathsf DL \overset{K}{\longrightarrow} \operatorname{sign}(L) \otimes \mathsf DL \overset{\operatorname{id} \otimes K}{\longrightarrow} \operatorname{sign}(L)^{\otimes 2}  \otimes \mathsf DL = \mathsf DL;
\]
\item[\checkmark] $G$ is \emph{positive definite}, i.e.~for all $D \in \mathsf DL\smallsetminus 0$, $G(D,D)$ is a positive element in $|L|$;
\item[\checkmark] $\Omega = G(-, K -)$, where, according to our free use of natural operations on $\mathsf DL$-tensors induced by natural operations on homogeneous tensors, we are extending $G$ to $\operatorname{sign}(L) \otimes \mathsf DL$ in the obvious way, and we are also using that $\operatorname{sign}(L) \otimes |L| = L$ (Remark \ref{rem:|-|,sign}).
\end{itemize}

The triple $(K, G, \Omega)$ is equivalent to ``almost contact metric''-like data. To explain this, first of all, recall from \cite[Section 6.3]{TVY20} that $G$ is equivalent to a tuple $(\nabla, \varphi, g_M, \nu)$ where
\begin{itemize}
\item[] $\nabla$ is a linear connection in $L$,
\item[] $\varphi \colon |L| \cong \mathbb R_M$ is an orientation preserving trivialization,
\item[] $g_M$ is a Riemannian metric on $M$,
\item[] $\nu$ is a $1$-form on $M$,
\end{itemize}
with no further restrictions. The relationship between $(\nabla, \varphi, g_M, \nu)$ and $G$ is the following:
\begin{itemize}
\item[] $\nabla (TM)$ is the $G$-orthogonal complement of $\bbI \in \calD (L)$,
\item[] $\varphi G(\bbI, \bbI) = 1_M$ (the constant section equal to $1$),
\item[] $g_M (X, Y) = \varphi G(\nabla_X, \nabla_Y)$ for all $X, Y \in \mathfrak X (M)$,
\item[] $ \nu (X) = \varphi \nabla_X \big(G(\bbI, \bbI)\big)$, 
\end{itemize}
where we are also denoting by $\nabla$ the induced connection in $|L|$.

\begin{remark}\label{rem:induced_connections}
Remember from Remark \ref{rem:L-conn}, that $\nabla$ induces connections on all other associated line bundles $f(L)$, also denoted $\nabla$ when $\dot f \neq 0$. However, beware that the trivialization $\varphi \colon |L| \to \bbR_M$ does not intertwine $\nabla$ with the canonical flat connection $\nabla^{can}$. It rather intertwines $\nabla$ with $\nabla^{can} + \nu$, i.e., for all $f \in C^\infty (M)$, and all $X \in \mathfrak X (M)$, we have
\[
\varphi \nabla_X (\varphi^{-1} f ) = X(f) +  \nu(X)f.
\]
It follows that the curvature of $\nabla$ is exactly $d\nu$, i.e.~for all $X, Y \in \mathfrak X(M)$ and all $\lambda \in \Gamma (L)$ we have
\[
\nabla_X\nabla_Y \lambda -\nabla_Y \nabla_X \lambda - \nabla_{[X,Y]}\lambda = d\nu (X,Y) \lambda.
\]
\end{remark}

\begin{remark}
Actually, $G$ (hence $\widetilde G$) is completely determined by the triple $(\varphi, g_M, \nu)$ \cite[Section 6.3]{TVY20}. However, the connection $\nabla$ (and the induced connections in all associated bundles) is often useful, so we prefer to keep it among the relevant data.
\end{remark}

We can now decompose $K$ using the isomorphism $\mathsf DL \cong \bbR_M \oplus TM$ induced by $\nabla$. Namely, once $G$ has been fixed, $K$ is equivalent to a tuple $(\phi, \xi, \eta, \gamma)$ where
\begin{itemize}
\item[] $\phi$ is a vector bundle map $\phi \colon TM \to \operatorname{sign}(L) \otimes TM$,
\item[] $\xi$ is a section of $\operatorname{sign}(L) \otimes TM$,
\item[] $\eta$ is a section of $\operatorname{sign}(L) \otimes T^\ast M$,
\item[] $\gamma$ is a section of $\operatorname{sign}(L)$.
\end{itemize}
The relationship between $(\phi, \xi, \eta, \gamma)$ and $K$ is the following:
\begin{itemize}
\item[] $K \nabla_X = \nabla_{\phi X} +  \eta(X)  \bbI$,
\item[] $K \bbI = -\nabla_\xi + \gamma  \bbI$,
\end{itemize}
for all $X \in \mathfrak X (M)$, where $\nabla \colon TM \to \mathsf DL$ is extended to $|L| \otimes TM$ in the obvious way.
Moreover, from $K^2 = -1$, the tuple $(\phi, \xi, \eta, \gamma)$ must satisfy 
\begin{itemize}
\item[] $\phi^2 = -1 + \xi \otimes \eta$,
\item[] $\phi \xi = - \gamma \xi$,
\item[] $\eta \circ \phi  = - \gamma  \eta$,
\item[] $\eta(\xi) = 1- \gamma^2$,
\end{itemize}
where, as usual, for every line bundle $L'$, we identify $|L'|^2$ with $\bbR_M$ and we also extend all vector bundle maps $F\colon E \to E'$ to tensor products of the type $L' \otimes E \to L' \otimes E'$ as $\operatorname{id} \otimes F$. In what follows we will understand all these extensions.

We get more identities from the condition that $G(-, K-)$ is alternating. Namely, for all $X, Y \in \mathfrak X (M)$,
\[
\begin{aligned}
0 & = G(\nabla_X, K \nabla_Y) + G(\nabla_Y, K \nabla_X) \\
& = G(\nabla_X, \nabla_{\phi Y}) + G (\nabla_Y, \nabla_{\phi X}) + \eta (Y) \otimes \cancel{G(\nabla_X, \bbI)} + \eta (X) \otimes \cancel{G (\nabla_Y, \bbI)} \\
& = \varphi^{-1} \big( g_M (X, \phi Y) + g_M (Y, \phi X) \big)
\end{aligned}
\]
(where we used that $\bbI$ is $G$-orthogonal to the image of $\nabla$), whence
\[
g_M (X, \phi Y) + g_M (Y, \phi X) = 0.
\]
Moreover, for all $X \in \mathfrak X(M)$
\[
\begin{aligned}
0 & = G(\nabla_X, K \bbI) + G(\bbI, K \nabla_X) \\
& = -G\big(\nabla_X, \nabla_\xi) + \gamma \otimes \cancel{G\big(\nabla_X, \bbI \big)} + \cancel{G\big(\bbI, \nabla_{\phi X} \big)} +  \eta(X)  \otimes  G\big(\bbI, \bbI \big) \\
& = \varphi^{-1} \big( -g_M (X, \xi) + \eta(X) \big),
\end{aligned}
\]
whence
\[
g_{M} (\xi , - ) =  \eta.
\]
Finally
\[
0  = G(\bbI, K \bbI) 
 = -\cancel{G(\bbI, \nabla_\xi)} + \gamma \otimes G(\bbI, \bbI) = \varphi^{-1} \big( \gamma \big),
\]
whence
\[
\gamma = 0.
\]
Now, we take care of $\Omega$, which is equivalent to a pair $(\alpha, \beta)$ where
\begin{itemize}
\item[] $\alpha \colon \wedge^2 TM \to L$ is an $L$-valued $2$-form on $M$, 
\item[] $\beta \colon TM \to L$ is an $L$-valued $1$-form on $M$,
\end{itemize}
with no further restrictions (see, e.g., \cite[Remark 2.7]{BVZ19}). Using the connection $\nabla$, we can express the relationship between $(\alpha, \beta)$ and $\Omega$ as follows:
\begin{itemize}
\item[] $\beta (X) = \Omega (\bbI, \nabla_X)$,
\item[] $\alpha (X, Y)  = \Omega (\nabla_X, \nabla_Y) -d^\nabla \beta (X, Y)$,
\end{itemize}
for all $X, Y \in \mathfrak X (M)$.  Now, from $\Omega = G(-, K-)$ we get, for all $X,Y \in \mathfrak X (M)$,
\[
\beta (X) = \Omega (\bbI, \nabla_X)  = G(\bbI, K\nabla_X) =  \varphi^{-1}\big( \eta (X) \big),
\]
whence
\[
\eta = \varphi \beta,
\]
and 
\[
\begin{aligned}
\alpha (X,Y) & =  \Omega (\nabla_X, \nabla_Y) +d^\nabla \beta (X, Y) \\
& = G(\nabla_X, K \nabla_Y) - d^\nabla \varphi^{-1} \eta (X, Y) \\
& = G(\nabla_X, \nabla_{\phi Y}) + \eta(Y) \otimes \cancel{G(\nabla_X, \bbI)} - d^\nabla \varphi^{-1} \eta (X, Y) \\
& = \varphi^{-1} \Big( g_M (X, \phi Y) - \varphi d^\nabla \varphi^{-1} \eta (X, Y)\Big) \\
& = \varphi^{-1} \Big( g_M (X, \phi Y) - \big(d^{can} \eta +  \nu \wedge \eta\big) (X, Y) \Big)
\end{aligned}
\]
where, in the last step, we also used Remark \ref{rem:induced_connections}. We conclude that
\[
g_M (-, \phi -) - d^{can} \eta -  \nu \wedge \eta = \varphi \alpha.
\]
The overall discussion proves the following

\begin{theorem}\label{theor:hom_ah}
Let $L \to M$ be a line bundle. The assignment 
\[
(\widetilde K, \widetilde G, \widetilde \Omega) \mapsto (\varphi, g_M, \nu, \phi, \xi, \eta, \alpha, \beta)
\] establishes a one-to-one correspondence between homogeneous almost Hermitian structures on the homogeneous manifold $\widetilde L \to M$ and tuples where 
\begin{itemize}
\item[] $\varphi \colon |L| \cong \bbR_M$ is an orientation preserving trivialization,
\item[] $g_M$ is a Riemannian metric on $M$,
\item[] $\nu$ is a $1$-form on $M$,
\item[] $\phi$ is a vector bundle map $TM  \to \operatorname{sign} (L) \otimes TM$,
\item[] $\xi$ is a $\operatorname{sign}(L)$-valued vector field on $M$,
\item[] $\eta$ is a  $\operatorname{sign}(L)$-valued $1$-form on $M$,
\item[] $\alpha$ is an $L$-valued $2$-form on $M$,
\item[] $\beta$ is an $L$-valued $1$-form on $M$,
\end{itemize}
satisfying the following identities
\begin{align}
\phi^2 & = -1 + \xi \otimes \eta \label{ACM1}\\
\phi \xi &= 0 \label{ACM2}\\
\eta \circ \phi & = 0 \label{ACM3}\\
\eta (\xi ) & = 1 \label{ACM4}\\
g_M (-, \phi -) + g_M (\phi -, -) & = 0 \label{ACM5}\\
g_M (\xi , -) & = \eta \label{ACM6}\\
\varphi \beta & = \eta \label{ACM7}\\
\varphi \alpha & = g_M (-, \phi -) - d^{can} \eta -  \nu \wedge \eta. \label{ACM8}
\end{align}
\end{theorem}

\begin{remark}
There are obvious redundancies in the statement of Theorem \ref{theor:hom_ah}. For instance $\alpha, \beta$ could be excluded from the tuple in the statement as they can be recovered from the other data.
\end{remark}

\begin{definition}\label{def:ACMDSH}
A tuple $(\varphi, g_M, \nu, \phi, \xi, \eta, \alpha, \beta)$ as in Theorem \ref{theor:hom_ah} will be referred to as a \textbf{almost contact metric dataset of homogeneity type H} (shortly, an \textbf{ACM dataset of type H}) on the line bundle $L \to M$ (or on $M$). 
\end{definition}

\begin{remark}\label{rem:triv_LB_1}
Let $L \to M$ be a line bundle and let $(\varphi, g_M, \nu, \phi, \xi, \eta, \alpha, \beta)$ be an ACM dataset of type H on it. If $L = \bbR_M$ is the trivial line bundle (hence $|L| = \operatorname{sign}(L) = \bbR_M$ as well) then $\varphi$ can be seen as a positive function on $M$. More importantly,
$\phi, \xi, \eta, \alpha, \beta$ are all ordinary tensors on $M$. In fact, in this case, the tuple $(\phi, \xi, \eta, g_M)$ is an ordinary \emph{almost contact metric structure} on $M$. Indeed, the identities \eqref{ACM1}, \eqref{ACM4} tell us that the triple $(\phi, \xi, \eta)$ is an \emph{almost contact structure}. Moreover, it easily follows from \eqref{ACM5}, \eqref{ACM1} and \eqref{ACM6} that
\[
g_M (\phi-, \phi-) = g + \eta \otimes \eta.
\]

This motivates the terminology in Definition \ref{def:ACMDSH}. Still the ACM dataset contains more information than the almost contact metric structure, the extra information consisting in $\varphi$ and the $1$-form $\nu$ (while $\alpha, \beta$ are determined by the other data).
\end{remark}

In the last part of this subsection we characterize the case $\nu = 0$ in terms of the Riemannian geometry of $\widetilde L$. This will be relevant in the next subsections.

\begin{lemma}
Let $L \to M$ be a line bundle, let $\widetilde G$ be a positive homogeneous Riemannian metric on $\widetilde L$, let $G \colon S^2 \mathsf DL \to |L|$ be the associated positive definite, symmetric $2$-form, and let $(\varphi, g_M, \nu)$ be the corresponding tuple. Then $\nu = 0$ if and only if the Euler vector field $\calE$ on $\widetilde L$ is pre-geodesic, i.e.~its orbits, that is the connected components of the fibers of the projection $\widetilde L \to M$, are unparameterized geodesics.
\end{lemma}

\begin{proof}
We denote by $\nabla^{LC}$ the Levi-Civita connection on $\widetilde L$ associated to the metric $\widetilde G$. Now, put $u = G(\bbI, \bbI) \in \Gamma (|L|)$. The homogenization of $u$ is $\widetilde u = \widetilde{G(\bbI, \bbI)} = \widetilde G (\calE, \calE)$. From the homogeneity condition on $\widetilde G$, we have $\calL_\calE \widetilde G = \widetilde G$. Hence, for all $U, V \in \mathfrak X (\widetilde L)$:
\[
\begin{aligned}
\widetilde G (U,V) & = (\calL_\calE \widetilde G)(U, V) \\
& = \calL_\calE \big( \widetilde G (U, V)\big) - \widetilde G \big([\calE , U], V\big) - \widetilde  G \big(U, [\calE , V]\big) \\
& = \widetilde G \big( \nabla^{LC}_\calE U, V \big) +  \widetilde G \big(  U, \nabla^{LC}_\calE V \big) - \widetilde G \big([\calE , U], V\big) - \widetilde  G \big(U, [\calE , V]\big) \\
& = \widetilde G \big( \nabla^{LC}_\calE U - [\calE , U], V \big)  + \widetilde G \big(  U, \nabla^{LC}_{\calE} V - [\calE , V]\big) \\
& = \widetilde G \big( \nabla^{LC}_U \calE, V \big) +  \widetilde G \big(  U, \nabla^{LC}_V \calE \big).
\end{aligned}
\]
In the particular case when $V = \calE$ and $U$ is orthogonal to $\calE$ we obtain
\[
0 = \widetilde G (U, \calE) = \widetilde G \big( \nabla^{LC}_U \calE, \calE \big) +  \widetilde G \big(  U, \nabla^{LC}_\calE \calE \big),
\]
i.e.
\begin{equation}\label{eq:1}
 \widetilde G \big(  U, \nabla^{LC}_\calE \calE \big) = - \widetilde G \big( \nabla^{LC}_U \calE, \calE \big) = \frac{1}{2} \calL_U  \widetilde G (\calE, \calE) = \frac{1}{2} \calL_U \widetilde u.
\end{equation}
Now, $\nu = 0$ if and only if $u$ is $\nabla$-parallel, if and only if $\widetilde u$ is constant along all directions orthogonal to $\calE$, which, in view of \eqref{eq:1}, is in turn equivalent to 
\[
\widetilde G \big(  U, \nabla^{LC}_\calE \calE \big) = 0
\]
for all vector fields $U$ orthogonal to $\calE$, in other words $\nabla^{LC}_\calE \calE$ is parallel to $\calE$.
\end{proof}

\begin{corollary}\label{cor:ACM}
Let $L \to M$ be an orientable and oriented  line bundle. The assignment 
\[
(\widetilde K, \widetilde G, \widetilde \Omega) \mapsto (\varphi ; \phi, \xi, \eta, g_M)
\] establishes a one-to-one correspondence between homogeneous almost Hermitian structures on $\widetilde L$ such that the Euler vector field is pre-geodesic and pairs consisting of an orientation preserving trivialization $\varphi: |L| \cong \bbR_M$ and an almost contact metric structure $(\phi, \xi, \eta, g_M)$ on $M$. 
\end{corollary}

\begin{proof}
Let $(\widetilde K, \widetilde G, \widetilde \Omega)$ be a homogeneous almost Hermitian structure on $\widetilde L$. If $L \to M$ is orientable and oriented, then $\operatorname{sign} (L) \to M$ can be canonically trivialized, and we can use the trivialization $\varphi$ to trivialize $L = \operatorname{sign} (L) \otimes |L|$ as well. In the next part of this proof we understand all these trivializations, so that $\phi, \eta, \xi, \alpha, \beta$ all become plain tensors on $M$ and, from Remark \ref{rem:triv_LB_1}, $(\phi, \xi, \eta, g_M)$ is an almost contact metric structure.

Conversely, let $(\phi, \xi, \eta, g_M)$ be an almost contact metric structure on $M$, then the identities \eqref{ACM2}, \eqref{ACM3}, \eqref{ACM6} easily follow. Finally, put $\nu = 0$ and use $(\phi, \xi, \eta, g_M)$ to define $\alpha, \beta$ via \eqref{ACM7}, \eqref{ACM8}. This concludes the proof.
\end{proof}

\begin{remark}
Assume $\nu = 0$ but $L$ non-trivial. In this case the orthogonal distribution $\calE^\bot$ to the Euler vector field on $\widetilde L$ is involutive, the leaves $\calO$ of $\calE^\bot$ are the level sets of the $|-|$-homogeneous function $\widetilde G (\calE, \calE)$ and they are double covers of $M$ under the projection $\pi : \widetilde L \to M$. The pull-back of $L$ to $\calO$ is trivial, hence $(\phi, \xi, \eta, g_M)$ can be pulled-back to a standard almost contact metric structure on $\calO$. The leaf $\calO : \widetilde G (\calE, \calE) = 1$ is a preferred choice (see also \cite{GGM25} for more details).
\end{remark}

Let $L \to M$ be a line bundle, let $(\widetilde K, \widetilde G, \widetilde \Omega)$ be a homogeneous almost Hermitian structure on $\widetilde L$. In the next subsections of this section we impose the standard integrability conditions on $(\widetilde K, \widetilde G, \widetilde \Omega)$ and express them in terms of the equivalent ACM dataset.

\subsection{Homogeneous Hermitian Structures}

The homogeneous almost Hermitian structure $(\widetilde K, \widetilde G, \widetilde \Omega)$ is \emph{Hermitian} when $\widetilde K$ is a complex structure, i.e.~the Nijenhuis torsion $N_{\widetilde K}$ of $\widetilde K$ vanishes identically. Remember that $N_{\widetilde K}$ is the vector valued $2$-form given by
\[
N_{\widetilde K} (U_1, U_2) = [\widetilde KU_1, \widetilde KU_2] + \widetilde K^2 [U_1, U_2] - \widetilde K \big([\widetilde K U_1, U_2] + [U_1, \widetilde K U_2] \big),
\]
for all $U_1, U_2 \in \mathfrak X (\widetilde L)$. Equivalently $(\widetilde K, \widetilde G, \widetilde \Omega)$ is an Hermitian structure if and only if $N_K = 0$ where $N_K \colon \wedge^2 \mathsf DL \to \mathsf D L$ is the $\mathsf DL$-valued $2$-form on $\mathsf DL$ given by 
\begin{equation}\label{eq:Atiyah_torsion}
N_{K} (D_1, D_2) = [KD_1, KD_2] + K^2 [D_1, D_2] - K \big([K D_1, D_2] + [D_1, K D_2] \big),
\end{equation}
for all $D_1, D_2 \in \Gamma (\mathsf DL)$. The square brackets $[D_1, D_2]$ in \eqref{eq:Atiyah_torsion} denote the Lie derivative $\calL_{D_1} D_2$ (see the beginning of Subsection \ref{subsec:AssLB}), and, as usual,
we extend $K \colon \mathsf D L \to \operatorname{sign}(L) \otimes \mathsf D L$ to vector bundle maps $K \colon f(L) \otimes \mathsf D L \to (f \operatorname{sign})(L) \otimes \mathsf D L$ in the obvious way.

Using the direct sum decomposition $\mathsf DL = TM \oplus \bbR_M$ given by the connection $\nabla$, we can decompose the identity $N_K = 0$ into $4$ identities. To do this remember first from Remark \ref{rem:induced_connections} that the curvature of the connection $\nabla$ is $d\nu$. Now, let $X, X_1, X_2 \in \mathfrak X (M)$ and specialize the condition $N_K (D_1, D_2) = 0$ to the following cases:

\begin{enumerate}

\item $D_1 = \nabla_{X_1}, D_2 = \nabla_{X_2}$: in this case, a direct computation using \eqref{ACM1} and  \eqref{ACM3} shows that 
\[
\begin{aligned}
& N_{K} (D_1, D_2) \\
& = N_{K} (\nabla_{X_1},\nabla_{X_2}) \\
& = \nabla_{N_\phi (X_1, X_2) + \big( d^{can}\eta (X_1, X_2) + d\nu (\phi X_1, X_2) + d\nu (X_1, \phi X_2) \big) \xi} \\
& \quad + \Big( d\nu (\phi X_1, \phi X_2) - d\nu (X_1, X_2) + d^{can}\eta (\phi X_1, X_2) + d^{can}\eta (X_1, \phi X_2)\Big)\bbI.
\end{aligned}
\]
We conclude that
\begin{equation*}
N_\phi + \xi \otimes \big(d^{can} \eta + d\nu (\phi -, -) + d\nu (-, \phi -)\big) = 0.
\end{equation*}
and
\begin{equation*}
d\nu (\phi -, \phi -) - d\nu  + d^{can}\eta (\phi -, -) + d^{can}\eta (-, \phi -) = 0.
\end{equation*}

\item $D_1 = \bbI, D_2 = \nabla_X$: in this case
\[
\begin{aligned}
N_{K} (D_1, D_2) & = N_{K} (\bbI, \nabla_X) \\
&  = -\nabla_{(\calL_\xi \phi)(X) + d\nu (\xi,X) \xi } - \big((\calL_\xi \eta)(X) + d \nu (\xi, \phi X)\big)\bbI.
\end{aligned}
\]
We conclude that
\begin{equation*}
\calL_\xi \phi + \xi \otimes \iota_\xi d\nu = 0,
\end{equation*}
and
\begin{equation*}
\calL_\xi \eta + (\iota_\xi d\nu) \circ \phi = 0.
\end{equation*}

\end{enumerate}

Summarizing, we have proved the following

\begin{theorem}\label{theor:NACM}
Let $L \to M$ be a line bundle. The assignment 
\[
(\widetilde K, \widetilde G, \widetilde \Omega) \mapsto (\varphi, g_M, \nu, \phi, \xi, \eta, \alpha, \beta)
\]
establishes a one-to-one correspondence between homogeneous Hermitian structures on $\widetilde L$ and ACM datasets of H-type satisfying the following integrability conditions:
\begin{align}
N_\phi + \xi \otimes \big(d^{can} \eta + d\nu (\phi -, -) + d\nu (-, \phi -)\big) &= 0 \label{NACM1}\\
d\nu (\phi -, \phi -) - d\nu  + d^{can}\eta (\phi -, -) + d^{can}\eta (-, \phi -) & = 0 \label{NACM2}\\
\calL_\xi \phi + \xi \otimes \iota_\xi d\nu & = 0 \label{NACM3}\\
\calL_\xi \eta + (\iota_\xi d\nu) \circ \phi &= 0. \label{NACM4}
\end{align}  
\end{theorem}
%

\begin{remark}\label{rem:depend_1}
The integrability conditions \eqref{NACM1}--\eqref{NACM4} are not independent. Namely, the last three are actually a consequence of the first one. Indeed, for all $X, Y \in \mathfrak X (M)$
\[
\begin{aligned}
\eta\big(N_\phi (X, Y))  & = \eta \left( [\phi X, \phi Y] \right) + \cancel{\eta \left(\phi^2 [X, Y] - \phi \big([\phi X, Y] + [X, \phi Y] \big) \right)} \\
& =  \eta \left( [\phi X, \phi Y] \right) \\
& = \cancel{\calL_{\phi X} \eta(\phi Y)} - \cancel{\calL_{\phi Y} \eta(\phi X)} - d^{can} \eta (\phi X ,\phi Y) \\
& = - d^{can} \eta (\phi X ,\phi Y),
\end{aligned}
\]
where we used \eqref{ACM3}. Now, suppose that \eqref{NACM1} holds true. Composing with $\eta$, and using \eqref{ACM4}, we find
\begin{align}
0 & = \eta\big(N_\phi (X, Y)\big) + d^{can} \eta (X, Y) + d\nu (\phi X, Y) + d\nu (X, \phi Y) \nonumber\\
& = - d^{can} \eta (\phi X, \phi Y) + d^{can} \eta (X, Y) + d\nu (\phi X, Y) + d\nu (X, \phi Y) \label{eq:2},
\end{align}
for all $X, Y \in \mathfrak X (M)$. Put $X = \xi$. Then, using \eqref{ACM2} we find
\begin{align}
0 & = - \cancel{d^{can} \eta (\phi \xi, \phi Y)} + d^{can} \eta (\xi, Y) + \cancel{d\nu (\phi \xi, Y)} + d\nu (\xi, \phi Y) \label{eq:4}\\
& = \big(\iota_\xi d^{can} \eta + (\iota_\xi d\nu) \circ \phi \big)( Y), \nonumber
\end{align}
for all $Y \in \mathfrak X (M)$, i.e.
\begin{equation}\label{eq:3}
\iota_\xi d^{can} \eta + (\iota_\xi d\nu) \circ \phi = 0.
\end{equation}

Using \eqref{ACM4}, from \eqref{eq:3} we also find
\[
\calL_\xi \eta + (\iota_\xi d\nu) \circ \phi = \iota_\xi d^{can} \eta + \cancel{d \big(\eta (\xi) \big) } +  (\iota_\xi d\nu) \circ \phi = 0,
\]
i.e.~the identity \eqref{NACM4} holds true. Moreover, putting $Y = \phi Z$ in \eqref{eq:4} we find
\[
\begin{aligned}
0 & = d^{can} \eta (\xi, \phi Z) + d\nu (\xi, \phi^2 Z) \\
& = d^{can} \eta (\xi, \phi Z) - d\nu (\xi, Z) + \eta(Z)  \cancel{d\nu (\xi, \xi)} \\
& = \big((\iota_\xi d^{can} \eta) \circ \phi - \iota_\xi d\nu \big) (Z),
\end{aligned}
\]
for all $Z \in \mathfrak X (M)$, i.e.
\begin{equation}\label{eq:5}
(\iota_\xi d^{can} \eta) \circ \phi - \iota_\xi d\nu = 0.
\end{equation}

Next, putting $X = X_1$ and $Y = \phi X_2$ in \eqref{eq:2} and using \eqref{eq:5} we find 
\[
\begin{aligned}
0 & = - d^{can} \eta (\phi X_1, \phi^2 X_2) + d^{can} \eta (X_1, \phi X_2) + d\nu (\phi X_1, \phi X_2) + d\nu (X_1, \phi^2 X_2) \\
& = d^{can} \eta (\phi X_1, X_2) + d^{can} \eta (X_1, \phi X_2) + d\nu (\phi X_1, \phi X_2) -d \nu (X_1, X_2) \\
 & \quad - \eta(X_2) \big(d^{can} \eta (\phi X_1, \xi) - d\nu (X_1, \xi)\big) \\
 & = d^{can} \eta (\phi X_1, X_2) + d^{can} \eta (X_1, \phi X_2) + d\nu (\phi X_1, \phi X_2) -d \nu (X_1, X_2) \\
 & \quad + \eta(X_2)  \cancel{\big( (\iota_\xi d^{can} \eta) \circ \phi - \iota_\xi d\nu \big)}(X_1) \\
 & = d^{can} \eta (\phi X_1, X_2) + d^{can} \eta (X_1, \phi X_2) + d\nu (\phi X_1, \phi X_2) -d \nu (X_1, X_2),
 \end{aligned}
\]
i.e.~the identity \eqref{NACM2} holds true. Finally, contracting \eqref{NACM1} with $\xi$ and $\phi X$,
\[
\begin{aligned}
0 & = N_\phi (\xi, \phi X) + \big(d^{can} \eta (\xi, \phi X) + \cancel{d\nu (\phi \xi, \phi X)} + d\nu (\xi, \phi^2 X)\big) \otimes \xi \\
& = \cancel{[\phi \xi, \phi^2 X]} + \phi^2 [\xi, \phi X] -\phi \big(\cancel{[\phi \xi, \phi X]} + [\xi, \phi^2 X]\big) + \cancel{\big( \iota_\xi d^{can} \eta + (\iota_\xi d\nu) \circ \phi \big)}(\phi X)  \xi \\
& = - [\xi, \phi X] + \eta\big([\xi, \phi X] \big)\xi + \phi [\xi,X] - \phi[\xi, \eta(X)  \xi] \\
& = - (\calL_\xi \phi)(X) + \eta\big([\xi, \phi X] \big)\xi - \calL_\xi \big(\eta (X)) \cancel{\phi \xi}  - \eta(X) \phi \cancel{[\xi,\xi]} \\
& = - (\calL_\xi \phi)(X) + \big( \calL_\xi \cancel{\eta(\phi X)} - \cancel{\calL_{\phi X} \eta(\xi)} - d^{can} \eta (\xi, \phi X) \big)\xi  \\
& = - \Big( \calL_\xi \phi + \xi \otimes \big((\iota_\xi d^{can} \eta) \circ \phi \big)\Big) (X) \\
& = - \big( \calL_\xi \phi + \xi \otimes \iota_\xi d \nu\big) (X) 
\end{aligned}
\]
for all $X \in \mathfrak X (M)$, where we used \eqref{ACM1}, \eqref{ACM2}, \eqref{ACM3}, \eqref{ACM4}, \eqref{eq:5}. In other words, \eqref{NACM3} holds true.
\end{remark}

\begin{corollary}\label{cor:NACM}
Let $L \to M$ be an orientable and oriented  line bundle. The assignment 
\[
(\widetilde K, \widetilde G, \widetilde \Omega) \mapsto (\varphi \,;\, \phi, \xi, \eta, g_M)
\]
establishes a one-to-one correspondence between homogeneous Hermitian structure on $\widetilde L$ such that the Euler vector field is pre-geodesic and pairs consisting of an orientation preserving trivialization $\varphi: |L| \cong \bbR_M$ and a normal almost contact metric structure $(\phi, \xi, \eta, g_M)$ on $M$. 
\end{corollary}

\begin{proof}
Let $(\widetilde K, \widetilde G, \widetilde \Omega)$ be a homogeneous Hermitian structure on $\widetilde L$. If $L$ is oriented then $(\widetilde K, \widetilde G, \widetilde \Omega)$ is equivalent to an almost contact metric structure $(\phi, \xi, \eta, g_M)$ on $M$ as in Corollary \ref{cor:ACM}. Moreover $\nu = 0$ and \eqref{NACM1} now says that $(\phi, \xi, \eta, g_M)$ is actually a normal almost contact metric structure. Conversely, let $(\phi, \xi, \eta, g_M)$ be a normal almost contact metric structure. Then, from Corollary \ref{cor:ACM} again, it comes from a homogeneous Hermitian structure with $\nu = 0$ on $\widetilde L$, the integrability condition on $(\phi, \xi, \eta, g_M)$ is exactly \eqref{NACM1}, and, from Remark \ref{rem:depend_1}, \eqref{NACM2}--\eqref{NACM4} follow.
\end{proof}


\subsection{Homogeneous Almost K\"ahler Structures}

In this subsection we impose an integrability condition on the symplectic structure $\widetilde \Omega$. Namely, let $(\widetilde K, \widetilde G, \widetilde \Omega)$ be a homogeneous almost Hermitian structure on $\widetilde L$. Remember that $(\widetilde K, \widetilde G, \widetilde \Omega)$ is an almost K\"ahler structure when $\widetilde \Omega$ is a symplectic structure, i.e.~$d\widetilde \Omega = 0$, or, equivalently, when $d_{\mathsf D L} \Omega = 0$. The latter identity is actually equivalent to $\alpha = 0$ (see \cite[Equation (17)]{BVZ19}), i.e.
\[
g_M (-, \phi -) - \nu \wedge \eta= d^{can}\eta.
\]
This proves the following

\begin{theorem}\label{theor:CM}
Let $L \to M$ be a line bundle. The assignment 
\[
(\widetilde K, \widetilde G, \widetilde \Omega) \mapsto (\varphi, g_M, \nu, \phi, \xi, \eta, \alpha, \beta)
\]
establishes a one-to-one correspondence between homogeneous almost K\"ahler structures on $\widetilde L$ and ACM datasets of type H additionally satisfying the following ``integrability'' condition:
\begin{equation}
g_M (-, \phi -) = d^{can}\eta + \nu \wedge \eta. \label{CM1}
\end{equation}  
\end{theorem}

\begin{remark}
Assume $\nu = 0$. In this case Equation \eqref{CM1} becomes $g_M (-, \phi -) = d^{can}\eta$ which should be compared with the usual compatibility between a metric $g_M$ and an almost contact structure $(\phi, \xi, \eta)$ defining a \emph{contact metric structure} (see, e.g., \cite{BG08}), namely
\[
2g_M (-, \phi -) = d\eta.
\]
The missing factor $2$ in \eqref{CM1} is only an apparent difference. Indeed by redefining 
\[
\begin{aligned}
\xi & \longrightarrow 2 \xi \\
\eta & \longrightarrow \tfrac{1}{2} \eta \\
g_M & \longrightarrow \tfrac{1}{4} g_M
\end{aligned}
\]
we can restore it. For this reason, in what follows, we will simply ignore this ``difference''.
\end{remark}

\begin{corollary}
Let $L \to M$ be an orientable and oriented line bundle. The assignment 
\[
(\widetilde K, \widetilde G, \widetilde \Omega) \mapsto (\varphi \,;\, \phi, \xi, \eta, g_M)
\]
establishes a one-to-one correspondence between homogeneous almost K\"ahler structures on $\widetilde L$ such that the Euler vector field is pre-geodesic and pairs consisting of an orientation preserving trivialization $\varphi : |L| \cong \bbR_M$ and a contact metric structure $(\phi, \xi, \eta, g_M)$ on $M$. 
\end{corollary}

\subsection{Homogeneous K\"ahler Structures}\label{subsec:HK}

Combining Theorems  \ref{theor:NACM} and \ref{theor:CM} we obtain

\begin{theorem}\label{theor:S}
Let $L \to M$ be a line bundle. The assignment 
\[
(\widetilde K, \widetilde G, \widetilde \Omega) \mapsto (\varphi, g_M, \nu, \phi, \xi, \eta, \alpha, \beta)
\] establishes a one-to-one correspondence between homogeneous K\"ahler structures on $\widetilde L$ and ACM datasets of type H additionally satisfying Equations \eqref{NACM1}--\eqref{NACM4} and \eqref{CM1} (equivalently Equations \eqref{NACM1} and \eqref{CM1}).
\end{theorem}

\begin{definition}
An ACM dataset of type H as in Theorem \ref{theor:S} will be referred to as a \textbf{Sasakian dataset}.
\end{definition}

\begin{corollary}
Let $L \to M$ be an orientable and oriented line bundle. The assignment 
\[
(\widetilde K, \widetilde G, \widetilde \Omega) \mapsto (\varphi \,;\, \phi, \xi, \eta, g_M)
\]
establishes a one-to-one correspondence between homogeneous K\"ahler structures on $\widetilde L$ such that the Euler vector field is pre-geodesic and pairs consisting of an orientation preserving trivialization $\varphi: |L| \cong \bbR_M$ and a Sasakian structure $(\phi, \xi, \eta, g_M)$ on $M$. 
\end{corollary}

\begin{remark}
As also remarked in \cite{GGM25}, in the K\"ahler case, $d \nu = 0$ does even imply $\nu = 0$. We here provide a different proof of this latter fact. Namely, let $(\widetilde G, \widetilde K, \widetilde \Omega)$ be a homogeneous K\"ahler structure on $\widetilde L$, and assume that $d \nu = 0$. Then, from \eqref{NACM4} we have $\calL_\xi \eta = 0$ and, as $\iota_\xi \eta = \eta (\xi) = 1$, we also have $\iota_\xi d^{can} \eta = 0$. Now, using $\eta \circ \phi = 0$, Equations \eqref{ACM3}, \eqref{ACM4} again and \eqref{CM1}, we get
\[
0 = \eta \circ \phi = g_M (\xi, \phi -) = \cancel{\iota_\xi d^{can} \eta} + \iota_\xi (\nu \wedge \eta) = \nu (\xi) \eta - \eta(\xi) \nu = \nu (\xi)  \eta - \nu,
\]
i.e.
\[
\nu = \nu (\xi) \eta.
\]
We want to show that $\nu (\xi) = 0$. By contradiction, let $x \in M$ be such that $\nu (\xi)_x \neq 0$. Then $\nu (\xi)_x \neq 0$ in a whole neighborhood $U$ of $x$ so that, in $U$, $\nu (\xi) \eta$, hence $\nu$, is a contact form (because the kernel of $\eta$ is a contact distribution), which contradicts $d \nu = 0$.
\end{remark}

\subsection{An Example}\label{subsec:example}

In \cite{GGM25} the authors construct an example of a homogeneous K\"ahler structure on a homogeneous manifold $\widetilde L$ associated to a non-trivial line bundle $L \to M$. In this section we focus on examples where ($L$ is trivial but) $\nu \neq 0$, which provide a generalization of Sasakian manifolds. In order to do this we first consider the following ingredients: 
\begin{itemize}
\item[] a line bundle $L \to M$,
\item[] its associated homogeneous manifold $\widetilde L$,
\item[] a section $s$ of the projection $\pi \colon \widetilde L \to M$,
\item[] the image $M_s \subseteq \widetilde L$ of $s$,
\item[] the diffeomorphism $\pi_s = \pi \circ i_{M_s} \colon M_s \to M$.
\end{itemize}
In this situation, as $\widetilde L$ possesses a section, then it is a trivial principal bundle (and actually there is a unique trivialization identifying $s$ with the constant function equal to $1$), hence $L, |L|, \operatorname{sign} (L)$ are trivial as well. Denote all by $\psi_s$ the trivializations induced by $s$. On sections they are given by
\[
\psi_s (\lambda) = s^\ast \widetilde \lambda \in C^\infty (M).
\]
where $\lambda$ is either a section of $L$, of $|L|$ or of $\operatorname{sign} (L)$.

Now, a homogeneous almost Hermitian structure $(\widetilde G, \widetilde K, \widetilde \Omega)$ on $\widetilde L$ determines an ACM dataset of type H $(\varphi, g_M, \nu, \phi, \xi, \eta, \alpha, \beta)$, and we can use the trivializations $\psi_s$ to interpret $\varphi$ as a positive function, and $\phi, \xi, \eta, \alpha, \beta$ as ordinary tensors on $M$. Moreover, using the diffeomorphism $\pi_s \colon M_s \to M$, we can pull-back $(\varphi, g_M, \nu, \phi, \xi, \eta, \alpha, \beta)$ to ordinary tensors $(\varphi_s, g_{M_s}, \nu_s, \phi_s, \xi_s, \eta_s, \alpha_s, \beta_s)$ on $M_s$ (clearly satisfying the same compatibilities). Our first aim in this section is expressing the tuple $(\varphi_s, g_{M_s}, \nu_s, \phi_s, \xi_s, \eta_s, \alpha_s, \beta_s)$ purely in terms of $(\widetilde G, \widetilde K, \widetilde \Omega)$, the Euler vector field $\calE$ on $\widetilde L$, and the inclusion $i_{M_s} \colon M_s \hookrightarrow \widetilde L$. We do this carefully in order not to loose track of the various identifications adopted.

First of all, notice that the Euler vector field $\calE$ is everywhere transverse to $M_s$, hence we have a direct sum decomposition $T\widetilde L|_{M_s} = TM_s \oplus \bbR \calE$. Denote by $\operatorname{pr}_{TM_s} \colon T\widetilde L|_{M_s} \to TM_s$ the projection with kernel $\bbR \calE$. We also denote $\widetilde \varphi = \widetilde G (\calE, \calE)^{-1}$ and compute:
\begin{enumerate}
\item $\varphi_s$:
\[
\varphi_s^{-1}  = \pi^\ast_s \Big( \psi_s \circ \varphi^{-1} (1_M) \Big) 
 = \pi^\ast_s \Big( s^\ast \widetilde{G(\bbI, \bbI)} \Big)
 = (s \circ \pi)^\ast \widetilde G (\calE, \calE) = \widetilde G (\calE, \calE)|_{M_s} = \widetilde \varphi {}^{-1}|_{M_s}.
\]
whence
\begin{equation}\label{eq:phi_s}
\varphi_s = \widetilde \varphi|_{M_s}.
\end{equation}
\item $g_{M_s}$: for all $X, Y \in \mathfrak X (M_s)$,
\[
\begin{aligned}
g_{M_s} (X, Y) & = \pi_s^\ast \Big(g_M (\pi_{s\ast} X, \pi_{s\ast} Y)\Big) = \pi_s^\ast \Big(\varphi G\big(\nabla_{\pi_{s\ast} X}, \nabla_{\pi_{s\ast} Y}\big) \Big) \\
& = \pi_s^\ast\Big(\varphi \psi_s^{-1} \psi_s G\big(\nabla_{\pi_{s\ast} X}, \nabla_{\pi_{s\ast} Y}\big) \Big) = \pi_s^\ast\Big(\varphi  \psi_s^{-1}\Big) \pi^\ast_s \Big( \psi_s G\big(\nabla_{\pi_{s\ast} X}, \nabla_{\pi_{s\ast} Y}\big) \Big) \\
& = \varphi_s \, \pi^\ast_s s^\ast \widetilde{G \big( \nabla_{\pi_{s\ast} X}, \nabla_{\pi_{s\ast} Y} \big)}  = \varphi_s \, \widetilde G \big( \widetilde{\nabla_{\pi_{s\ast} X}}, \widetilde{\nabla_{\pi_{s\ast} Y}} \big)|_{M_s},
\end{aligned}
\]
where we interpreted $\varphi \psi_s^{-1}$ as a function on $M$. In order to continue the computation, we have to study vector fields on $\widetilde L$ of the type $\widetilde{\nabla_{\pi_{s\ast} X}}$. Actually, it is enough to compute $\widetilde{\nabla_{\pi_{s\ast} X}}$ along $M_s$. In fact, $\widetilde{\nabla_{\pi_{s\ast} X}}$ is the unique vector field $H_X$ along $M_s$ satisfying i) $H_X$ is orthogonal to $\calE$, and ii) $H_X$ projects to $\pi_{s\ast}X$. So it must be the component of $X$ orthogonal to $\calE$, i.e.
\[
H_X = X - \widetilde G (\calE, \calE)^{-1} \calE_\flat (X) \calE|_{M_s} = X - \varphi_s  \, \calE_\flat (X) \calE|_{M_s},
\]
where $\calE_\flat := \widetilde G_\flat (\calE)$, whence
\[
\widetilde G \big( \widetilde{\nabla_{\pi_{s\ast} X}}, \widetilde{\nabla_{\pi_{s\ast} Y}} \big)|_{M_s} = \widetilde G (H_X, H_Y)  = \widetilde G (X, Y) - \varphi_s \, \calE_\flat (X) \,\calE_\flat (Y),
\]
and
\[
g_{M_s} (X, Y)   =\varphi_s \, \widetilde G \big( \widetilde{\nabla_{\pi_{s\ast} X}}, \widetilde{\nabla_{\pi_{s\ast} Y}} \big)|_{M_s} = \varphi_s \widetilde G (X, Y) - \varphi_s^2 \calE_\flat (X) \,\calE_\flat (Y),
\]
in other words
 \begin{equation}\label{eq:g_M_s}
g_{M_s} =   i^\ast_{M_s }\Big(\widetilde \varphi \, \widetilde G - \widetilde \varphi{}^2 \,  \calE_\flat \otimes  \calE_\flat \Big).
\end{equation}
\item $\nu_s$: for all $X \in \mathfrak X (M_s)$,
\[
\begin{aligned}
\nu_s (X) & = \pi^\ast_s \Big( \nu (\pi_{s \ast} X) \Big)  = \pi^\ast_s \Big( \varphi \nabla_{\pi_{s \ast} X} \big(G(\bbI, \bbI)\big) \Big) \\
& = \varphi_s \, \pi^\ast_s \Big( \psi_s \nabla_{\pi_{s \ast} X} \big(G(\bbI, \bbI)\big)\Big)  =\varphi_s \, \widetilde{\nabla_{\pi_{s \ast} X} \big(G(\bbI, \bbI)} |_{M_s} \\
& = \varphi_s \, \widetilde{\nabla_{\pi_{s \ast} X}} \big(\widetilde G (\calE, \calE) \big) = \varphi_s \, H_X \widetilde \varphi{}^{-1}  = \varphi_s \, d\widetilde \varphi{}^{-1} \big(X - \varphi_s  \, \calE_\flat (X) \calE|_{M_s}\big) \\
& = \Big(\varphi_s d \varphi_s^{-1} (X) - \varphi_s^2 \, \calE_\flat (X) \calE (\widetilde \varphi{}^{-1})\Big)|_{M_s}.
\end{aligned}
\]
But, from the homogeneity properties of $\widetilde G$ we have $\calE (\widetilde \varphi{}^{-1}) = \widetilde \varphi{}^{-1}$, hence
\[
\nu_s (X) =  \varphi_s \big( d\varphi_s^{-1} (X) - \calE_\flat (X) \big),
\]
i.e.
\[
\nu_s =  i^\ast_{M_s} \Big(\widetilde \varphi \, \big( d \widetilde \varphi{}^{-1}- \calE_\flat \big) \Big).
\]

\item $\xi_s$: we have
 \[
 \begin{aligned}
 \xi_s & = \operatorname{pr}_{TM_s} H_{\xi_s} = \operatorname{pr}_{TM_s} \widetilde{\nabla_{\pi_{s\ast} \xi_s}} |_{M_s} = \operatorname{pr}_{TM_s} \widetilde{\nabla_\xi} |_{M_s} = - \operatorname{pr}_{TM_s} \widetilde{K\bbI } |_{M_s} \\
 & = - \operatorname{pr}_{TM_s}  \widetilde K  \calE |_{M_s}.
 \end{aligned}
 \]
Summarizing
 \begin{equation}\label{eq:xi_s}
 \xi_s = - \operatorname{pr}_{TM_s}  \widetilde K  \calE |_{M_s}.
\end{equation}
 \item $\phi_s, \eta_s$: for all $X \in \mathfrak X (M_s)$ we have, on one side,
 \[
 \begin{aligned}
 \widetilde{K\nabla_{\pi_{s \ast} X} } |_{M_s}& = \widetilde K  \widetilde{\nabla_{\pi_{s \ast} X}} |_{M_s} = \widetilde K H_X = \widetilde K X - \varphi_s \, \calE_\flat (X) \widetilde K \calE |_{M_s},
 \end{aligned}
 \]
 and, on the other side,
 \[
  \begin{aligned}
 \widetilde{K\nabla_{\pi_{s \ast} X} |_{M_s}} & = \widetilde{\nabla_{\phi \pi_{s \ast} X} + \eta (\pi_{s \ast}) \otimes \bbI}  =\widetilde{\nabla_{\phi \pi_{s \ast} X}} |_{M_s} + \widetilde{\eta (\pi_{s \ast} X)} \calE |_{M_s} \\
 & =\widetilde{\nabla_{\pi_{s \ast} \phi_s X}} |_{M_s} + \widetilde{\psi_s^{-1} \pi_{s\ast} \eta_s (X)} \calE |_{M_s}  = H_{\phi_s X} + \pi_s^\ast s^\ast \Big(\widetilde{\psi_s^{-1} \pi_{s\ast} \eta_s (X)} \Big)\calE |_{M_s} \\
 & = H_{\phi_s X} + \pi_s^\ast \pi_{s\ast} \eta_s (X)\calE |_{M_s} = H_{\phi_s X} + \eta_s (X)\calE |_{M_s} \\
 & = \phi_s X + \big(\eta_s (X) - \varphi_s \, \calE_\flat (\phi_s X) \big)\calE |_{M_s},
   \end{aligned}
 \]
 so that
 \[
 \phi_s X + \big(\eta_s (X) - \varphi_s \, \calE_\flat (\phi_s X) \big)\calE |_{M_s}=  \widetilde K X - \varphi_s \, \calE_\flat (X) \widetilde K \calE |_{M_s}.
 \]
 Applying $\operatorname{pr}_{TM_s}$ (to both sides) we get
 \[
 \phi_s X = \operatorname{pr}_{TM_s}  \widetilde K X,
 \]
 hence
 \[
  \phi_s = \operatorname{pr}_{TM_s}  \widetilde K |_{TM_s},
 \]
 while applying $\calE_\flat$, and using that $\widetilde \Omega$ is alternating we get 
 \[
\varphi_s^{-1} \,  \eta_s (X)  = \widetilde G ( \calE, \widetilde K H_X) = \widetilde \Omega (\calE, H_X) = \widetilde \Omega (\calE, X),
 \]
hence
 \begin{equation}\label{eq:eta_s}
  \eta_s = i^\ast_{M_s} \Big( \widetilde \varphi  \, \iota_\calE \widetilde \Omega \Big).
 \end{equation}

Notice that Equations \eqref{eq:g_M_s}, \eqref{eq:xi_s} and \eqref{eq:eta_s} are duly compatible with \eqref{ACM6}, i.e.
\[
g_{M_s} (\xi_s, -) = \eta_s.
\]

\item $\alpha_s, \beta_s$: it follows from Equations \eqref{ACM7}, \eqref{ACM8} that
\begin{equation}\label{eq:beta_s}
\beta_s = - \varphi_s^{-1} \, \eta_s
\end{equation}
and
\[
\alpha_s = \varphi_s^{-1} i^\ast_{M_s}\Big( g_{M_s} (-, \phi_s-) - d \eta_s - \nu_s \wedge \eta_s \Big).
\]
\end{enumerate}

Now let $(P, \widetilde G, \widetilde K, \widetilde \Omega)$ be a K\"ahler manifold, let $\calE \in \mathfrak X (P)$ be a vector field such that
\begin{equation}\label{eq:inf_hom}
\calL_\calE \widetilde G = \widetilde G, \quad \calL_\calE \widetilde K = 0, \quad \text{whence} \quad \calL_\calE \widetilde \Omega = \widetilde \Omega,
\end{equation}
and let $\Sigma \subseteq P$ be a hypersurface everywhere transverse to $\calE$. In this situation, locally around $\Sigma$, we can identify $P$ with a homogeneous manifold such that 1) $\calE$ is the Euler vector field, 2) $\Sigma = M_s$ is the image of a section $s$ of $P$, and 3) $(\widetilde G, \widetilde K, \widetilde \Omega)$ is a homogeneous K\"ahler structure on $P$. It follows that $\Sigma$ inherits a Sasakian dataset $(\varphi_s, g_{M_s}, \nu_s, \phi_s, \xi_s, \eta_s, \beta_s)$ uniquely determined by $\calE$ and $(\widetilde G, \widetilde K, \widetilde \Omega)$ via Formulas \eqref{eq:phi_s}--\eqref{eq:beta_s}.

\begin{example}\label{exmpl1}
It is easy to produce an example where the orthogonal distribution to $\calE$ is not involutive, so that $\nu_s \neq 0$. For instance, on $P = \mathbb R^4 \smallsetminus 0$ with standard coordinates $(x_1, y_1, x_2, y_2)$ consider the standard K\"ahler structure $(\widetilde G, \widetilde K, \widetilde \Omega)$ given by 
\[
\widetilde G = dx_1^2  + dy_1^2 + dx_2^2 + dy_2^2,
\]
\[
\widetilde K \frac{\partial}{\partial x_1} = \frac{\partial}{\partial y_1}, \quad \text{and} \quad \widetilde K \frac{\partial}{\partial x_2} = \frac{\partial}{\partial y_2},
\]
so that
\[
\widetilde \Omega = dy_1 \wedge dx_1 + dy_2 \wedge dx_2.
\]
Consider also the vector field
\[
\calE = \frac{1}{2} \left( x_1 \frac{\partial}{\partial x_1} +  y_1 \frac{\partial}{\partial y_1} +x_2 \frac{\partial}{\partial x_2} + y_2 \frac{\partial}{\partial y_2} \right),
\]
and the unit $3$-sphere $\Sigma = S^3$. Equations \eqref{eq:inf_hom} are satisfied. Moreover $\varphi_s = 4$ and
\[
\nu_s = -  i^\ast_{S^3} \calE_\flat =  - 2  i^\ast_{S^3} \left(x_1 dx_1 + y_1 dy_1 + x_2 dx_2 + y_2 dy_2 \right) = 0.
\]
It follows that $(\phi_s, \xi_s, \eta_s, g_{M_s})$ is an ordinary Sasakian structure on $S^3$ (it is actually the standard Sasakian structure on $S^3$). However, by slightly deforming $\calE$, we can produce a new example with $\nu_s \neq 0$. Namely, consider the unitary vector field
\[
U = -y_1 \frac{\partial}{\partial x_1} +  x_1 \frac{\partial}{\partial y_1} +y_2 \frac{\partial}{\partial x_2} - x_2 \frac{\partial}{\partial y_2}.
\]
As $U$ is an infinitesimal symmetry of the K\"ahler structure, redefining $\calE \to \calE + U$ preserves Equations \eqref{eq:inf_hom}. However, the orthogonal complement to $\calE$ is not integrable anymore. Indeed, in this case,
\[
\varphi_s =  \frac{4}{3}
\]
is still constant, but 
\[
\nu_s = - i^\ast_{S^3} U_\flat = \frac{1}{2}i^\ast_{S^3} \Big(y_1 dx_1 - x_1 dy_1 - y_2 dx_2 + x_2 dy_2 \Big) \neq 0.
\]
\end{example}

\section{Invariant K\"ahler Manifolds}\label{sec:3}

In this section we study the \emph{invariant case} (Case D of Subsection \ref{subsec:AHSonHM}). We parallel the presentation in Section \ref{sec:3}. We omit the details when they are similar to those therein.

\subsection{Invariant Almost Hermitian Structures}\label{subsec:IAHS} Let $L \to M$ be a line bundle. An \textbf{invariant almost Hermitian structure} on $\widetilde L$ is an almost Hermitian structure $(\widetilde K, \widetilde G, \widetilde \Omega)$ on $\widetilde L$ satisfying the following homogeneity conditions (CASE D in Subsection \ref{subsec:AHSonHM}):
\begin{enumerate}
\item $h_r^\ast \widetilde K = \widetilde K$,
\item $h_r^\ast \widetilde G = \widetilde G$,
\item $h_r^\ast \widetilde \Omega = \widetilde \Omega$,
\end{enumerate}
for all $r \in \bbR^\times$. It follows that $\widetilde K, \widetilde G, \widetilde \Omega$ are the homogenizations of 
\begin{enumerate}
\item a vector bundle map $K \colon \mathsf DL \to \mathsf DL$,
\item a non-degenerate, symmetric $2$-form $G \colon S^2 \mathsf DL \to \bbR_M$,
\item a non-degenerate, alternating $2$-form $\Omega \colon \wedge^2 \mathsf DL \to \bbR_M$,
\end{enumerate}
satisfying
\begin{itemize}
\item[] $K^2 = -1$,
\item[] $G$ is positive definite,
\item[] $\Omega = G(- , K-)$.
\end{itemize}
Similarly as in Subsection \ref{subsec:2A}, the triple $(K, G, \Omega)$ is equivalent to a new type of ``almost contact metric''-like data. To see this, notice that $G$ is equivalent to a triple $(\nabla, u, g_M)$ where
\begin{itemize}
\item[] $\nabla$ is a linear connection in $L$,
\item[] $u \in C^\infty (M)$ is a positive function on $M$,
\item[] $g_M$ is a Riemannian metric on $M$,
\end{itemize}
with no further restrictions. The relationship between $(\nabla, u, g_M)$ and $G$ is as follows:
\begin{itemize}
\item[] $\nabla (TM)$ is the $G$-orthogonal complement of $\bbI \in \calD (L)$,
\item[] $u = G(\bbI, \bbI)$,
\item[] $g_M (X, Y) = G(\nabla_X, \nabla_Y)$ for all $X, Y \in \mathfrak X (M)$.
\end{itemize}

Decomposing $K$ via the isomorphism $\mathsf DL \cong \bbR_M \oplus TM$ induced by $\nabla$, we find that (once $G$ has been fixed) $K$ is equivalent to a tuple $(\phi, \xi, \eta, \gamma)$ where now
\begin{itemize}
\item[] $\phi$ is a $(1,1)$-tensor,
\item[] $\xi$ is a vector field,
\item[] $\eta$ is a $1$-form,
\item[] $\gamma$ is a smooth function,
\end{itemize}
on $M$. The relationship between $(\phi, \xi, \eta, \gamma)$ and $K$ is the following:
\begin{itemize}
\item[] $K \nabla_X = \nabla_{\phi X} + \eta (X) \bbI$,
\item[] $K \bbI = -\nabla_\xi + \gamma \bbI$,
\end{itemize}
for all $X \in \mathfrak X (M)$. Moreover, from $K^2 = -1$ we get
\begin{itemize}
\item[] $\phi^2 = -1 + \xi \otimes \eta$,
\item[] $\phi \xi = - \gamma \xi$,
\item[] $\eta \circ \phi = - \gamma \eta$,
\item[] $ \eta(\xi) = 1 - \gamma^2$.
\end{itemize}
From the condition that $G(-, K-)$ is alternating we get 
\begin{itemize}
\item[] $g_M (-,\phi -) + g_M (\phi -, -) = 0$,
\item[] $g_M(\xi, -) = u\,\eta$,
\item[] $\gamma = 0$.
\end{itemize}
Using $\nabla$ we can also decompose $\Omega$, and find that it is equivalent to a pair $(\alpha, \beta)$, where
\begin{itemize}
\item[] $\alpha$ is a $2$-form,
\item[] $\beta$ is a $1$-form,
\end{itemize}
on $M$. The pair $(\alpha, \beta)$ is related to $\Omega$ via
\begin{itemize}
\item[] $\beta (X) = \Omega (\bbI, \nabla_X)$,
\item[] $\alpha (X, Y) = \Omega (\nabla_X, \nabla_Y)$,
\end{itemize}
for all $X, Y \in \mathfrak X (M)$. From $\Omega = G(-, K -)$ we get,
\[
\beta (X) = u \, \eta (X), \quad \text{i.e.} \quad \beta = u \, \eta,
\]
and
\[
\alpha (X, Y)  = g_M (X, \phi Y), \quad \text{i.e.} \quad \alpha = g_M(-, \phi -) .
\]
Summarizing, we have the following

\color{black}

\begin{theorem}\label{theor:inv_ah}
Let $L \to M$ be a line bundle. The assignment
\[
(\widetilde K, \widetilde G, \widetilde \Omega) \mapsto (\nabla, u, g_M, \phi, \xi, \eta, \alpha, \beta)
\]
establishes a one-to-one correspondence between invariant almost Hermitian structures on $\widetilde L \to M$ and tuples where 
\begin{itemize}
\item[] $\nabla$ is a linear connection in $L$,
\item[] $u$ is a positive function on $M$,
\item[] $g_M$ is a Riemannian metric on $M$,
\item[] $\phi$ is $(1,1)$-tensor on $M$,
\item[] $\xi$ is a vector field on $M$,
\item[] $\eta$ is a  $1$-form on $M$,
\item[] $\alpha$ is a $2$-form on $M$,
\item[] $\beta$ is a $1$-form on $M$,
\end{itemize}
satisfying the following identities
\begin{align}
\phi^2 & = -1 + \xi \otimes \eta \label{ACM'1}\\
\phi \xi &= 0 \label{ACM'2}\\
\eta \circ \phi & = 0 \label{ACM'3}\\
\eta (\xi ) & = 1 \label{ACM'4}\\
g_M (-, \phi -) + g_M (\phi -, -) & = 0 \label{ACM'5}\\
g_M (\xi , -) & = u \,\eta \label{ACM'6}\\
\beta & = u \, \eta \label{ACM'7}\\
\alpha & = g_M (-, \phi -). \label{ACM'8}
\end{align}
\end{theorem}

\begin{definition}
A tuple $(\nabla, u, g_M, \phi, \xi, \eta, \alpha, \beta)$ as in Theorem \ref{theor:inv_ah} will be referred to as an \textbf{almost contact metric dataset of homogeneity type I} (shortly, an \textbf{ACM dataset of type I}) on $L \to M$ (or on $M$).
\end{definition}

\begin{corollary}\label{cor:ACM'}
Let $L \to M$ be a line bundle. The assignment
\[
(\widetilde K, \widetilde G, \widetilde \Omega) \mapsto (\nabla \,;\, u\,;\, \phi, \xi, \eta, g_M)
\]
establishes a one-to-one correspondence between invariant almost Hermitian structures on $\widetilde L$ and triples consisting of a connection $\nabla$ in $L$, a positive function $u$, and an almost contact metric structure $(\phi, \xi, \eta, g_M)$ on $M$. 
\end{corollary}

\begin{proof}
Let $(\nabla, u, g_M, \phi, \xi, \eta, \alpha, \beta)$ be an ACM dataset of type I on $L \to M$. Then $(\phi, \xi, \eta, g_M)$ is an almost contact metric structure on $M$. Moreover $\alpha, \beta$ are determined by the other data.
\end{proof}

\subsection{Invariant Hermitian Structures}

Now assume that $\widetilde K$ is a complex structure. The (vanishing) Nijenhuis torsion of $K$ can be decomposed again in the directions parallel and $G$-orthogonal to $\bbI$ yielding integrability conditions for $(\phi, \xi, \eta)$. The latter conditions involve the curvature of the connection $\nabla$ which is a closed $2$-form on $M$ which we denote $\rho \in \Omega^2 (M)$. Namely, from $N_K (\nabla_{(-)}, \nabla_{(-)})=0$ we get
\[
N_\phi + \xi \otimes \big( d \eta + \rho (\phi -, -) + \rho (-, \phi-)\big) = 0,
\] 
and
\[
\rho (\phi -, \phi -) - \rho + d \eta (\phi -, -) + d\eta (-, \phi -) = 0,
\]
and, from $N_K (\bbI, \nabla_{(-)})=0$ we get
\[
\calL_\xi \phi + \xi \otimes \iota_\xi \rho = 0,
\]
and
\[
\calL_\xi \eta + (\iota_\xi \rho ) \circ \phi = 0.
\]
Summarizing, we have the following

\begin{theorem}\label{theor:NACM'}
Let $L \to M$ be a line bundle. The assignment
\[
(\widetilde K, \widetilde G, \widetilde \Omega) \mapsto (\nabla, u, g_M, \phi, \xi, \eta, \alpha, \beta)
\]
establishes a one-to-one correspondence between invariant Hermitian structures on $\widetilde L$ and ACM datasets of type I additionally satisfying the following integrability conditions:
\begin{align}
N_\phi + \xi \otimes \big(d \eta + \rho (\phi -, -) + \rho (-, \phi -)\big)&= 0 \label{NACM'1}\\
\rho (\phi -, \phi -) - \rho  + d\eta (\phi -, -) + d\eta (-, \phi -) & = 0 \label{NACM'2}\\
\calL_\xi \phi + \xi \otimes \iota_\xi \rho & = 0 \label{NACM'3}\\
\calL_\xi \eta + (\iota_\xi \rho) \circ \phi &= 0. \label{NACM'4}
\end{align}  
\end{theorem}

\begin{remark}
The integrability condition \eqref{NACM'1} does actually imply \eqref{NACM'2}--\eqref{NACM'4} as in Remark \ref{rem:depend_1}. 
\end{remark}

\begin{corollary}
Let $L \to M$ be a line bundle. The assignment
\[
(\widetilde K, \widetilde G, \widetilde \Omega) \mapsto (\nabla\,;\,  u\,;\,  \phi, \xi, \eta, g_M)
\]
establishes a one-to-one correspondence between invariant Hermitian structures on $\widetilde L$ such that the orthogonal distribution to the Euler vector field is involutive and triples consisting of a flat linear connection $\nabla$ in $L$, a positive function $u$, and a normal almost contact metric structure on $M$. 
\end{corollary}

\begin{proof}
The orthogonal distribution to $\calE$ is the principal connection in $\widetilde L$ corresponding to the connection $\nabla$ in $L$, whence the former is involutive if and only if $\nabla$ is flat if and only if $\rho = 0$. The rest is obvious.
\end{proof}

\subsection{Invariant Almost K\"ahler Structures}

Let $(\widetilde K, \widetilde G, \widetilde \Omega)$ be an invariant almost Hermitian structure on $\widetilde L$ and assume that $(\widetilde K, \widetilde G, \widetilde \Omega)$ is actually an almost K\"ahler structure, i.e.~$d \widetilde \Omega = 0$, equivalently $d_{\mathsf DL} \Omega = 0$. Decomposing the latter equation into the parallel and the orthogonal directions to $\bbI$ we get the following

\begin{theorem}\label{theor:AcK}
Let $L \to M$ be a line bundle. The assignment
\[
(\widetilde K, \widetilde G, \widetilde \Omega) \mapsto (\nabla, u, g_M, \phi, \xi, \eta, \alpha, \beta)
\]
establishes a one-to-one correspondence between invariant almost K\"ahler structures on $\widetilde L$ and ACM dataset of type I additionally satisfying the following ``integrability'' conditions:
\begin{align}
d\beta & = 0 \label{AcK2} \\
d \alpha & = \rho \wedge \beta \label{AcK1}
\end{align}  
\end{theorem}

 We now establish the precise relationship between invariant K\"ahler structures and almost co-K\"ahler structures.
 
 \begin{lemma}
Let $L \to M$ be a line bundle, let $\widetilde G$ be an invariant Riemannian metric on $\widetilde L$, let $G \colon S^2 \mathsf D L \to \bbR_M$ be the associated positive definite, symmetric $2$-form, and let $(\nabla, u, g_M)$ be the corresponding tuple. The following conditions are equivalent:
\begin{enumerate}
\item $u$ is constant,
\item $\calE$ is pre-geodesic,
\item $\calE$ is geodesic.
\end{enumerate} 
\end{lemma}

\begin{proof}
Denote by $\pi : \widetilde L \to M$ the projection. Then $\pi^\ast u = \widetilde{G (\bbI, \bbI)} = \widetilde G (\calE, \calE)$. From the invariance condition on $\widetilde G$, we have $\calL_\calE \widetilde G = 0$, i.e.~$\calE$ is a Killing vector field. Hence
\[
\begin{aligned}
Y ( \pi^\ast u)  = Y \big(\widetilde G (\calE, \calE) \big)  = 2 \widetilde G (\nabla^{LC}_Y \calE, \calE) = - 2 \widetilde G (Y, \nabla^{LC}_{\calE}\calE),
\end{aligned}
\]
for all $Y \in \mathfrak X (\widetilde L)$, where we used the Killing equation. This shows that \emph{(1)} if and only if \emph{(2)}. It remains to show that \emph{(2)} $\Rightarrow$ \emph{(3)}. So assume there exists a smooth function $h$ on $\widetilde L$ such that $\nabla^{LC}_\calE \calE = h \calE$. From the Killing equation
\[
0 =  G (\calE, \nabla^{LC}_{\calE}\calE) = h\, \pi^\ast u,
\]
whence $h = 0$.
\end{proof}

\begin{lemma}
Let $L \to M$ be a line bundle, let $\widetilde G$ be an invariant Riemannian metric on $\widetilde L$, let $G \colon S^2 \mathsf D L \to \bbR_M$ be the associated positive definite, symmetric $2$-form, and let $(\nabla, u, g_M)$ be the corresponding tuple. The following conditions are equivalent:
\begin{enumerate}
\item $u$ is constant and $\nabla$ is flat,
\item $\calE$ is parallel.
\end{enumerate} 
\end{lemma}

\begin{proof}
Being $\calE$ Killing, condition \emph{(1)} is equivalent to $\calE$ being geodesic and, simultaneously, the orthogonal distribution being involutive. This is in turn equivalent to $\calE$ being parallel. Indeed, for a Killing vector field $\calE$ being parallel is equivalent to $d \calE_\flat = 0$, where $\calE_\flat = \widetilde G_\flat (\calE)$. Now, for all $X, Y \in \mathfrak X (\widetilde L)$,
\[
d \calE_\flat (X,Y) = X\big( \widetilde G (\calE, Y) \big) - Y (\widetilde G(\calE, X) \big) - \widetilde G \big(\calE, [X, Y]\big).
\]
We distinguish two cases
\begin{itemize}
\item[\checkmark] $X = \calE$ and $Y$ is the horizontal lift of a vector field on $M$ via $\nabla$. Then
\begin{equation}\label{eq:a1}
d \calE_\flat (X,Y) = -Y (\pi^\ast u);
\end{equation}
\item[\checkmark] both $X, Y$ are the horizontal lifts of vector fields on $M$ via $\nabla$. Then
\begin{equation}\label{eq:a2}
d \calE_\flat (X,Y) = - \widetilde G \big(\calE, [X, Y]\big).
\end{equation}
\end{itemize}
So, assume Condition \emph{(1)} in the statement holds. Then both the right hand sides of \eqref{eq:a1} and \eqref{eq:a2} vanish whence $\calE$ is parallel. Conversely, if Condition \emph{(1)} in the statement holds then the left hand sides of \eqref{eq:a1} and \eqref{eq:a2} vanish. From the second equation we conclude that $\nabla$ is flat, while from the first one we conclude that $u$ is constant.
\end{proof}

\begin{corollary}
Let $L \to M$ be a line bundle. The assignment
\[
(\widetilde K, \widetilde G, \widetilde \Omega) \mapsto (\nabla \, ;\, u \,;\, \phi, \xi, \eta, u^{-1}g_M)
\]
establishes a one-to-one correspondence between invariant almost K\"ahler structures on $\widetilde L$ such that the Euler vector field is parallel and triples consisting of a flat linear connection $\nabla$ in $L$, a positive function $u$, and an almost co-K\"ahler structure on $M$. 
\end{corollary}

\begin{proof}
When the Euler vector field is parallel, $u$ is constant and the connection $\nabla$ is flat, i.e.~$\rho = 0$. In this case, $d \alpha = 0$ and, from $d \beta = 0$, we also get $d \eta = 0$, i.e. $(\phi, \xi, \eta, u^{-1}g_M)$ is a co-K\"ahler structure on $M$.
\end{proof}

\subsection{Invariant K\"ahler Structures}

Combining Theorems  \ref{theor:NACM'} and \ref{theor:AcK} we obtain

\begin{theorem}\label{theor:cK}
Let $L \to M$ be a line bundle. The assignment
\[
(\widetilde K, \widetilde G, \widetilde \Omega) \mapsto (\nabla, u, g_M, \phi, \xi, \eta, \alpha, \beta)
\]
establishes a one-to-one correspondence between invariant K\"ahler structures on $\widetilde L$ and ACM dataset of type I additionally satisfying Equations \eqref{NACM'1}--\eqref{NACM'4} and \eqref{AcK1}, \eqref{AcK2} (equivalently Equations \eqref{NACM'1}, \eqref{AcK1}, \eqref{AcK2}).
\end{theorem}

\begin{definition}
An ACM dataset of type I as in Theorem \ref{theor:cK} will be referred to as a \textbf{co-K\"ahler dataset}.
\end{definition}

\begin{corollary}
Let $L \to M$ be a line bundle. The assignment
\[
(\widetilde K, \widetilde G, \widetilde \Omega) \mapsto (\nabla \, ;\, u \,;\, \phi, \xi, \eta, u^{-1}g_M)
\]
establishes a one-to-one correspondence between invariant K\"ahler structures on $\widetilde L$ such that the Euler vector field is parallel and triples consisting of a flat linear connection $\nabla$ in $L$, a constant positive function $u$ and a co-K\"ahler structure on $M$. 
\end{corollary}

\begin{remark}
Let $L$ be a line bundle, and let $(\widetilde K, \widetilde G, \widetilde \Omega)$ be an invariant K\"ahler structure on $\widetilde L$. Let $u$ be constant, and $\rho = 0$. Then $u$ can be actually set to $1$ without great loss of generality. It is enough to redefine
\[
\begin{aligned}
\widetilde G & \longrightarrow \widetilde G (\calE, \calE)^{-1} \, \widetilde G, \quad \text{(whence $u \longrightarrow 1$)}\\
\widetilde \Omega & \longrightarrow \widetilde G (\calE, \calE)^{-1} \, \widetilde \Omega,
\end{aligned}
\]
which amounts to redefine
\[
\begin{aligned}
g_M & \longrightarrow u^{-1} g_M, \\
\alpha & \longrightarrow u^{-1}\alpha, \\
\beta & \longrightarrow u^{-1}\beta.
\end{aligned}
\]
All the integrability conditions are obviously preserved so the new $(\widetilde K, \widetilde G, \widetilde \Omega)$ is still an invariant K\"ahler structure and the new $(\phi, \xi, \eta, g_M)$ is a co-K\"ahler structure.
\end{remark}

\subsection{Two Examples}

Similarly as in Subsection \ref{subsec:example}, here we provide two example of an invariant K\"ahler structure where ($L$ is trivial but) the connection $\nabla$ is not flat, so that its curvature $\rho$ enters the ``integrabilities'' in a non-trivial way. In the first example $u$ is not constant, while in the second example $u$ is actually constant. This will provide even locally a genuine  generalization of co-K\"ahler manifolds. So, let $L, \widetilde L, s, M_s, \pi_s, \psi_s$ be as in Subsection \ref{subsec:example}. An invariant almost Hermitian structure $(\widetilde G, \widetilde K, \widetilde \Omega)$ on $\widetilde L$ determines an ACM dataset of type I $(\nabla, u, g_M, \nu, \phi, \xi, \eta, \beta)$ on $M$ and, using $\pi_s$ and $\psi_s$, we can pull-back the latter to an ACM dataset of type I $(\nabla^s, u_s, g_{M_s}, \nu_s, \phi_s, \xi_s, \eta_s, \alpha_s, \beta_s)$ on $M_s$, where $\nabla^s$ is now a connection in the trivial line bundle $\bbR_{M_s} \to M_s$ and the rest are ordinary tensors on $M_s$. We also denote by $\rho_s$ the curvature of $\nabla^s$. Arguing in a very similar way as in Subsection \ref{subsec:example} we find
\begin{equation}\label{eq:u_s}
u_s = \widetilde u |_{M_s}, \quad \text{where} \quad \widetilde u =  \widetilde G(\calE, \calE).
\end{equation}
Moreover,
\begin{equation}
\nabla^s = \nabla^{can} + \nu_s, \quad \text{and} \quad \rho_s =  d\nu_s,
\end{equation}
where $\nabla^{can}$ is the canonical flat connection in $\bbR_{M_s}$ and $\nu_s \in \Omega^1 (M_s)$ is the $1$-form given by
\begin{equation}
\nu_s = i_{M_s}^\ast \left( \widetilde u{}^{-1} \calE_\flat \right).
\end{equation}
Additionally
 \begin{align}
g_{M_s} &=  i^\ast_{M_s }\Big( \widetilde G - \widetilde u^{-1}  \calE_\flat \otimes  \calE_\flat \Big)\label{eq:g_M_s'}, \\
 \xi_s & = - \operatorname{pr}_{TM_s}  \widetilde K  \calE |_{M_s}, \\
   \phi_s & = \operatorname{pr}_{TM_s}  \widetilde K |_{TM_s}, \\
    \eta_s &= i^\ast_{M_s} \Big( \widetilde u{}^{-1}  \, \iota_\calE \widetilde \Omega \Big)\label{eq:eta_s'}, \\
    \beta_s & = u_s \,  \eta_s \\
    \alpha_s &= u_s \, g_{M_s}(- ,\phi_s -) \label{eq:beta_s'}.
\end{align}

To find a concrete example let $(P, \widetilde G, \widetilde K, \widetilde \Omega)$ be a K\"ahler manifold, let $\calE \in \mathfrak X (P)$ be a vector field such that
\begin{equation}\label{eq:inf_inv}
\calL_\calE \widetilde G = 0, \quad \calL_\calE \widetilde K = 0, \quad \text{whence} \quad \calL_\calE \widetilde \Omega = 0,
\end{equation}
and let $\Sigma \subseteq P$ be a hypersurface everywhere transverse to $\calE$. In this situation, locally around $\Sigma$, we can identify $P$ with a principal $\mathbb R^\times$-bundle such that 1) $\calE$ is the Euler vector field on $P$, 2) $\Sigma = M_s$ is the image of a section $s$ of $P$, and 3) $(\widetilde G, \widetilde K, \widetilde \Omega)$ is an invariant K\"ahler structure on $P$. It follows that $\Sigma$ inherits a co-K\"ahler dataset $(\nabla^s, u_s, g_{M_s}, \phi_s, \xi_s, \eta_s, \beta_s)$ uniquely determined by $\calE$ and $(\widetilde G, \widetilde K, \widetilde \Omega)$ via Formulas \eqref{eq:u_s}--\eqref{eq:beta_s'}.

\begin{example}
It is easy to produce an example where $u_s$ is not constant, nor is $\nabla^s$ flat. For instance, on $P = \mathbb R^4 \smallsetminus 0$ with standard coordinates $(x_1, y_1, x_2, y_2)$ consider the standard K\"ahler structure $(\widetilde G, \widetilde K, \widetilde \Omega)$. Consider also the vector field
\[
\calE = \frac{\partial}{\partial y_2},
\]
and the coordinate hyperplane $\Sigma : y_2 = 0$. Equations \eqref{eq:inf_inv} are satisfied. Moreover $u_s = 1$ and
\[
\nu_s = - i^\ast_{\Sigma} \calE_\flat =  - i^\ast_{\Sigma} (dy_2) = 0 \quad \Rightarrow \quad \rho_s = 0.
\]
It follows that $(g_{M_s}, \phi_s, \xi_s, \beta_s, \alpha_s)$ is an ordinary co-K\"ahler structure  on $\Sigma = \bbR^3$. Deforming $\calE$, we can produce a new example with non constant $u_s$ and non-zero $\rho_s$. Namely, consider again the unitary vector field $U$ from Example \ref{exmpl1}. Redefining $\calE \to \calE + U$ preserves Equations \eqref{eq:inf_inv}. However, in this case
\[
u_s = 1 - x_2 + x_1^2 + x_2^2 + y_1^2
\]
while the orthogonal complement to $\calE$ is clearly not integrable anymore. We leave the simple details to the reader.
\end{example} 

\begin{example}
We conclude with an example where $u_s$ is constant but $\rho_s \neq 0$. In $\bbR^4$ with standard coordinates $(x_1, y_1, x_2, y_2)$ consider the open submanifold $P : 2x_1 > x_2^2 + y_2^2$. On $P$ consider the metric
\[
\widetilde G = dx_1^2 + dy_1^2 + 2x_1 (dx_2^2 + dy_2^2) + 2 (x_2 dx_2 + y_2 dy_2)dx_1 + 2(x_2dy_2 - y_2 dx_2)dy_1,
\]
the standard complex structure $\widetilde J$, and the symplectic form
\[
\widetilde \Omega = - dx_1 \wedge dy_1 - 2x_1 dx_2 \wedge dy_2 + y_2 dx_1 \wedge dx_2 - x_2 dx_1 \wedge dy_2 +x_2dy_1 \wedge dx_2 + y_2 dy_1 \wedge dy_2.
\]
A direct computation shows that $(\widetilde G, \widetilde K, \widetilde \Omega)$ is a K\"ahler structure on $P$. Consider also the vector field
\[
\calE = \frac{\partial}{\partial y_1},
\]
and the hypersurface $\Sigma \subseteq P$ given by $\Sigma : y_1 = 0$. Equations \eqref{eq:inf_inv} are satisfied. Moreover $u_s = 1$ but the orthogonal complement to $\calE$ is the kernel of the $1$-form $\theta = dy_1 - y_2 dx_2 + x_2 dy_2$ so that $\theta \wedge d\theta = 2 dy_1 \wedge dx_2 \wedge dy_2 \neq 0$, whence $\calE^\bot$ in non-integrable. In fact
\[
\rho_s = 4 dx_2 \wedge dy_2 \neq 0.
\]
\end{example}

\appendix

\section{Case C}\label{appendix}
 In this appendix we state without proof an analogue of Theorems \ref{theor:inv_ah}, \ref{theor:NACM'}, \ref{theor:AcK}, and \ref{theor:cK} in the very similar CASE C of Subsection \ref{subsec:AHSonHM}. So, let $L \to M$ be a line bundle, and let $(\widetilde K, \widetilde G, \widetilde \Omega)$ be an almost contact metric structure on $\widetilde L$ satisfying the homogeneity properties of CASE C. Then $\widetilde K, \widetilde G, \widetilde \Omega$ are the homogenization of a vector bundle map $K : \mathsf DL \to \operatorname{sign} (L) \otimes \mathsf D L$ satisfying $K^2 = -1$, a positive definite, symmetric $2$-form $G : S^2 \mathsf DL \to \bbR_M$, and a non-degenerate, alternating $2$-form $\Omega : \wedge^2 \mathsf DL \to L$ such that $\Omega = G (-, K-)$. In their turn, $(G, K, \Omega)$ determine a tuple $(\nabla, u, g_M, \phi, \xi, \eta, \alpha, \beta)$ where
 \begin{itemize}
 \item[] $\nabla$ is a linear connection in $L$,
 \item[] $u$ is a positive smooth function on $M$,
 \item[] $g_M$ is a Riemannian metric on $M$,
 \item[] $\phi$ is a vector bundle map $TM \to \operatorname{sign}(L) \otimes TM$,
 \item[] $\xi$ is a $\operatorname{sign}(L)$-valued vector field on $M$,
 \item[] $\eta, \beta$ are both $\operatorname{sign} (L)$-valued $1$-forms on $M$,
 \item[] $\alpha$ is a $\operatorname{sign} (L)$-valued $2$-form on $M$,
 \end{itemize}
 satisfying Equations \eqref{ACM'1}--\eqref{ACM'8}. The relationship between $(\nabla, u, g_M, \phi, \xi, \eta, \alpha, \beta)$ and $(G, K, \Omega)$ is formally identical to the relationship between $(G, K, \Omega)$ and its associated ACM dataset of type I in the invariant case (Subsection \ref{subsec:IAHS}). We call $(\nabla, u, g_M, \phi, \xi, \eta, \alpha, \beta)$ an \textbf{ACM dataset of type C} on $L$, and denote by $\rho \in \Omega^2 (M)$ the curvature of $\nabla$.
 
 \begin{theorem}
 The assignment
 \[
 (\widetilde K, \widetilde G, \widetilde \Omega) \mapsto (\nabla, u, g_M, \phi, \xi, \eta, \alpha, \beta)
 \]
 establishes a one-to-one correspondence between almost Hermitian structures on $\widetilde L$ with the homogeneity properties of CASE C and ACM datasets of type C on $L$. Moreover
 \begin{enumerate}
 \item $(\widetilde K, \widetilde G, \widetilde \Omega)$ is an Hermitian structure if and only if $(\phi, \xi, \eta)$ satisfy Equations \eqref{NACM'1}--\eqref{NACM'4} (equivalently \eqref{NACM'1});
 \item $(\widetilde K, \widetilde G, \widetilde \Omega)$ is an almost K\"ahler structure if and only if $(\alpha, \beta)$ satisfy Equations \eqref{AcK2}--\eqref{AcK1} (where, however, the de Rham differential $d$ must be replaced by the connection differential $d^{can}$ associated to the canonical flat connection $\nabla^{can}$ in $\operatorname{sign}(L)$);
 \item $(\widetilde K, \widetilde G, \widetilde \Omega)$ is a K\"ahler structure if and only if $(\phi, \xi, \eta, \alpha, \beta)$ satisfy both the conditions in (1) and (2).
 \end{enumerate}
 \end{theorem}
 Notice that when $L$ is a trivial line bundle, then the invariant case and CASE C are identical.

\end{document}